\documentclass[english,11pt]{amsart}
\usepackage[latin9]{inputenc}
\usepackage{amstext}
\usepackage{amsthm}
\usepackage{amssymb}

\makeatletter

\newcommand{\lyxmathsym}[1]{\ifmmode\begingroup\def\b@ld{bold}
  \text{\ifx\math@version\b@ld\bfseries\fi#1}\endgroup\else#1\fi}

\numberwithin{equation}{section}
\numberwithin{figure}{section}
\theoremstyle{plain}
\newtheorem{thm}{\protect\theoremname}[section]
  \theoremstyle{definition}
  \newtheorem{defn}[thm]{\protect\definitionname}
  \theoremstyle{remark}
  \newtheorem{rem}[thm]{\protect\remarkname}
  \theoremstyle{plain}
  \newtheorem{cor}[thm]{\protect\corollaryname}
  \theoremstyle{plain}
  \newtheorem{lem}[thm]{\protect\lemmaname}
  \theoremstyle{remark}
  \newtheorem*{rem*}{\protect\remarkname}

\makeatother

\usepackage{babel}
  \providecommand{\corollaryname}{Corollary}
  \providecommand{\definitionname}{Definition}
  \providecommand{\lemmaname}{Lemma}
  \providecommand{\remarkname}{Remark}
\providecommand{\theoremname}{Theorem}

\begin{document}

\title{Wave Equations with Moving Potentials}

\author{Gong Chen}

\date{\today}

\email{gc@math.toronto.edu}

\address{Department of Mathematics, The University of Toronto, 540 St. George
St., Room 6290, Toronto, Ontario, Canada, M5S 2E4.}

\begin{abstract}
In this paper, we study the some reversed Strichartz estimates along
general time-like trajectories for wave equations in $\mathbb{R}^{3}$.
Some applications of the reversed Strichartz estimates and the structure
of wave operators to the wave equation with one potential are also
discussed. These techniques are useful to analyze the stability problem
of traveling solitons.
\end{abstract}

\maketitle

\section{Introduction\label{sec:Intro}}

Our starting point is the free wave equation ($H_{0}=-\Delta$) on
$\mathbb{R}^{3}$
\begin{equation}
\partial_{tt}u-\Delta u=0
\end{equation}
with initial data
\begin{equation}
u(x,0)=g(x),\,u_{t}(x,0)=f(x).
\end{equation}
We can write down $u$ explicitly,
\begin{equation}
u=\frac{\sin\left(t\sqrt{-\Delta}\right)}{\sqrt{-\Delta}}f+\cos\left(t\sqrt{-\Delta}\right)g.
\end{equation}
It obeys the energy inequality,
\begin{equation}
E_{F}(t)=\int_{\mathbb{R}^{3}}\left|\partial_{t}u(t)\right|^{2}+\left|\nabla u(t)\right|^{2}\,dx\lesssim\int_{\mathbb{R}^{3}}\left|f\right|^{2}+\left|\nabla g\right|^{2}\,dx.
\end{equation}
 We also have the well-known dispersive estimates for the free wave
equation on $\mathbb{R}^{3}$:
\begin{equation}
\left\Vert \frac{\sin\left(t\sqrt{-\Delta}\right)}{\sqrt{-\Delta}}f\right\Vert _{L^{\infty}\left(\mathbb{R}^{3}\right)}\lesssim\frac{1}{\left|t\right|}\left\Vert \nabla f\right\Vert _{L^{1}\left(\mathbb{R}^{3}\right)},\label{eq:disper1}
\end{equation}

\begin{equation}
\left\Vert \cos\left(t\sqrt{-\Delta}\right)g\right\Vert _{L^{\infty}\left(\mathbb{R}^{3}\right)}\lesssim\frac{1}{\left|t\right|}\left\Vert \Delta g\right\Vert _{L^{1}\left(\mathbb{R}^{3}\right)}.\label{eq:disper2}
\end{equation}
For the sake of completeness, the proofs of estimates \eqref{eq:disper1}
and \eqref{eq:disper2} are provided in details in Appendix A. (Notice
that the estimate \eqref{eq:disper2} is slightly different from the
estimates commonly used in the literature, such as Krieger-Schlag
\cite{KS} where one needs the $L^{1}$ norm of $D^{2}g$ instead
of $\Delta g$).

\smallskip

Strichartz estimates can be derived abstractly from these dispersive
inequalities and the energy inequality. With some appropriate $\left(p,q,s\right)$,
one has
\begin{equation}
\|u\|_{L_{t}^{p}L_{x}^{q}}\lesssim\|g\|_{\dot{H}^{s}}+\|f\|_{\dot{H}^{s-1}}\label{eq:IFStri}
\end{equation}
The non-endpoint estimates for the wave equations can be found in
Ginibre-Velo \cite{GV}. Keel\textendash Tao \cite{KT} also obtained
sharp Strichartz estimates for the free wave equation in $\mathbb{R}^{n},\,n\geq4$
and everything except the endpoint in $\mathbb{R}^{3}$. See Keel-Tao
\cite{KT} and Tao's book \cite{Tao} for more details on the subject's
background and the history. 

\smallskip

In $\mathbb{R}^{3}$, there is no hope to obtain such  an estimate
with the $L_{t}^{2}L_{x}^{\infty}$ norm, the so-called endpoint Strichartz
estimate for free wave equations, cf.~Klainerman-Machedon \cite{KM}
and Machihara-Nakamura-Nakanishi-Ozawa \cite{MNNO}. But if we reverse
the order of space-time integration, one can obtain a version of reversed
Strichartz estimates from the Morawetz estimate, cf.~Theorem \ref{thm:EndRStrichF}:
\begin{equation}
\left\Vert \frac{\sin\left(t\sqrt{-\Delta}\right)}{\sqrt{-\Delta}}f\right\Vert _{L_{x}^{\infty}L_{t}^{2}}\lesssim\left\Vert f\right\Vert _{L^{2}\left(\mathbb{R}^{3}\right)},\,\left\Vert \cos\left(t\sqrt{-\Delta}\right)g\right\Vert _{L_{x}^{\infty}L_{t}^{2}}\lesssim\left\Vert g\right\Vert _{\dot{H}^{1}\left(\mathbb{R}^{3}\right)}.
\end{equation}
These estimates are extended to inhomogeneous cases and perturbed
Hamiltonian in Beceanu-Goldberg \cite{BecGo}. In Section \ref{sec:revered},
we will study these estimates and their generalizations intensively.
We will also study the other extreme case with the norm $L_{x}^{6}L_{t}^{\infty}$
as in Beceanu-Goldberg \cite{BecGo}:
\begin{equation}
\left\Vert \frac{\sin\left(t\sqrt{-\Delta}\right)}{\sqrt{-\Delta}}f\right\Vert _{L_{x}^{6}L_{t}^{\infty}}\lesssim\left\Vert f\right\Vert _{L^{2}\left(\mathbb{R}^{3}\right)},\,\left\Vert \cos\left(t\sqrt{-\Delta}\right)g\right\Vert _{L_{x}^{6}L_{t}^{\infty}}\lesssim\left\Vert g\right\Vert _{\dot{H}^{1}\left(\mathbb{R}^{3}\right)}.
\end{equation}
These two estimates can be combined together to remedy the failure
of the regular endpoint Strichartz estimate. For example, in Chen
\cite{GC3}, these estimates are used to study the multi-soliton solution
to a wave equation in which case, each soltion decays slowly.

\smallskip

Next, we consider a linear wave equation with a real-valued stationary
potential,
\begin{equation}
H=-\Delta+V,
\end{equation}
\begin{equation}
\partial_{tt}u+Hu=\partial_{tt}u-\Delta u+Vu=0,
\end{equation}
\begin{equation}
u(x,0)=g(x),\,u_{t}(x,0)=f(x).
\end{equation}
Explicitly, we have
\begin{equation}
u=\frac{\sin\left(t\sqrt{H}\right)}{\sqrt{H}}f+\cos\left(t\sqrt{H}\right)g.
\end{equation}
For the class of short-range potentials we consider in this paper,
under our hypotheses $H$ only has pure absolutely continuous spectrum
on $[0,\infty)$ and a finite number of negative eigenvalues. It is
crucial to notice that if there is a negative eigenvalue $E<0$, the
associated eigenfunction responds to the wave equation propagators
with a scalar factor by $\cos\left(t\sqrt{E}\right)$ or $\frac{\sin\left(t\sqrt{E}\right)}{E^{\frac{1}{2}}}$,
both of which will grow exponentially since $\sqrt{E}$ is purely
imaginary. Thus, Strichartz estimates for $H$ must include a projection
$P_{c}$ onto the continuous spectrum in order to get away from this
situation. 

\smallskip

The problem of the dispersive decay and Strichartz estimates for the
wave equation with a potential has received much attention in recent
years, see the papers by Beceanu-Goldberg \cite{BecGo}, Krieger-Schlag
\cite{KS} and the survey by Schlag \cite{Sch} for further details
and references. 

The Strichartz estimates for the perturbed wave equations are in the
form:
\begin{equation}
\left\Vert \frac{\sin\left(t\sqrt{H}\right)}{\sqrt{H}}P_{c}f+\cos\left(t\sqrt{H}\right)P_{c}g\right\Vert _{L_{t}^{p}L_{x}^{q}}\lesssim\|g\|_{\dot{H}^{1}}+\|f\|_{L^{2}}
\end{equation}
with $2<p,\,\frac{1}{2}=\frac{1}{p}+\frac{3}{q}.$ One also has the
endpoint reversed Strichartz estimates:
\begin{equation}
\left\Vert \frac{\sin\left(t\sqrt{H}\right)}{\sqrt{H}}P_{c}f+\cos\left(t\sqrt{H}\right)P_{c}g\right\Vert _{L_{x}^{\infty}L_{t}^{2}}\lesssim\|f\|_{L^{2}}+\|g\|_{\dot{H}^{1}},
\end{equation}
see Theorem \ref{thm:PStriRStrich}. For the other extreme case, we
have
\begin{equation}
\left\Vert \frac{\sin\left(t\sqrt{H}\right)}{\sqrt{H}}P_{c}f+\cos\left(t\sqrt{H}\right)P_{c}g\right\Vert _{L_{x}^{6}L_{t}^{\infty}}\lesssim\|f\|_{L^{2}}+\|g\|_{\dot{H}^{1}}.
\end{equation}

In Section \ref{sec:Prelim} and Section \ref{sec:revered}, we will
systematically pass the estimates for free equations to the perturbed
case via the structure formula of wave operators. This strategy also
works in many other contexts provided that the free solution operators
commute with translations and reflections.

\smallskip

For wave equations in $\mathbb{R}^{3}$, there are several difficulties.
For example, the failure of the $L_{t}^{2}L_{x}^{\infty}$ estimate
and the weakness of decay power $\frac{1}{t}$ in dispersive estimates.
The reversed Strichartz estimates might circumvent these difficulties.
Reversed Strichartz estimates along time-like trajectories play an
important role in the analysis of wave equations of moving potentials.
For example, in \cite{GC2}, we used some preliminary versions of
these estimates to show Strichartz estimates for wave equations with
charge transfer Hamiltonian. 

There are extra difficulties when dealing with time-dependent potentials.
For example, given a general time-dependent potential $V(x,t)$, it
is not clear how to introduce an analog of bound states and a spectral
projection. The evolution might not satisfy group properties any more.
It might also result in the growth of certain norms of the solutions,
see Bourgain's book \cite{Bou}. 

\smallskip

The second part of this paper, we apply the endpoint reversed Strichartz
estimates along trajectories to study the wave equation with one moving
potential:
\begin{equation}
\partial_{tt}u-\Delta u+V\left(x-\vec{Y}(t)\right)u=0
\end{equation}
which appears naturally in the study of stability problems of traveling
solitons. We impose that the trajectories are asymptotic to straight
lines as in \cite{Graf}.

\smallskip

For Schr\"odinger equations with moving potentials, one can find
references and progress, for example in Beceanu-Soffer \cite{BS},
Rodnianski-Schlag-Soffer \cite{RSS}. Compared with Schr\"odinger
equations, wave equations have some natural difficulties, for example
the evolution of bound states of wave equations leads to exponential
growth meanwhile the evolution of bound states of Schr\"odinger equations
is merely multiplied by oscillating factors. We also notice that Lorentz
transformations are space-time rotations, therefore one can not hope
to succeed by the approach used with Schr\"odinger equations based
on Galilei transformations. The geometry becomes much more complicated
in the wave equation context. A crucial step to study wave equations
with moving potentials is to understand the change of the energy under
Lorentz transformations. In Chen \cite{GC2}, we obtained that the
energy stays comparable under Lorentz transformations. In this paper,
we study this by a different approach based on local energy conservation
which requires less decay of the potential. As a byproduct, we also
obtain Agmon's estimates for the decay of eigenfunctions associated
to negative eigenvalues of $H$. 

\subsection{Main results}
\begin{defn}[Admissible trajectories]
A trajectory $\vec{Y}(t)\in\mathbb{R}^{3}$
is said to be admissible if $\vec{Y}(t)$ is $C^{1}$ and there exists
$0\leq\ell<1$ such $\left|\vec{Y}'(t)\right|\leq\ell<1$ for $t\in\mathbb{R}$. 
\end{defn}

Consider the solution to the free wave equation ($H_{0}=-\Delta$),
\begin{equation}
u(x,t)=\frac{\sin\left(t\sqrt{-\Delta}\right)}{\sqrt{-\Delta}}f+\cos\left(t\sqrt{-\Delta}\right)g+\int_{0}^{t}\frac{\sin\left(\left(t-s\right)\sqrt{-\Delta}\right)}{\sqrt{-\Delta}}F(s)\,ds
\end{equation}
and let $\vec{Y}(t)\in\mathbb{R}^{3}$ be an admissible trajectory.
Setting
\begin{equation}
u^{S}(x,t):=u\left(x+\vec{Y}(t),t\right),
\end{equation}
we estimate
\begin{equation}
\sup_{x\in\mathbb{R}^{3}}\int\left|u^{S}(x,t)\right|^{2}dt
\end{equation}
in terms of the initial energy and various norms of $F$. The idea
behind these estimates is that the fundamental solution of the free
wave equation is supported on the light cone. Along a time-like curve,
the propagation will only meet the light cone once. 
\begin{thm}
\label{thm:reversedF}Let $\vec{Y}(t)$ be an admissible trajectory.
First of all, for the standard case, one has
\begin{equation}
\left\Vert u\right\Vert _{L_{x}^{\infty}L_{t}^{2}}\lesssim\|f\|_{L^{2}}+\|g\|_{\dot{H}^{1}}+\left\Vert F\right\Vert _{L_{x}^{\frac{3}{2},1}L_{t}^{2}}.
\end{equation}
Along the trajectory, we have
\begin{equation}
\left\Vert u^{S}(x,t)\right\Vert _{L_{x}^{\infty}L_{t}^{2}}\lesssim\|f\|_{L^{2}}+\|g\|_{\dot{H}^{1}}+\left\Vert \nabla F\right\Vert _{L_{x}^{1}L_{t}^{2}}.\label{eq:freeesti11}
\end{equation}
If $\vec{Y}(t)$ does not change the direction, then
\begin{equation}
\left\Vert u^{S}(x,t)\right\Vert _{L_{x}^{\infty}L_{t}^{2}}\lesssim\|f\|_{L^{2}}+\|g\|_{\dot{H}^{1}}+\left\Vert F\right\Vert _{L_{d}^{1}L_{\widehat{d}}^{2,1}L_{t}^{2}},\label{eq:freeesti12}
\end{equation}
where $d$ is the direction of $\vec{Y}(t)$ and $\hat{d}$ is the
subspace orthogonal to the direction $d$.

Let $\vec{Z}(t)$ be another admissible trajectory, we have the same
estimates as above with $F$ replaced by
\begin{equation}
F^{S'}(x,t):=F\left(x+\vec{Z}(t),t\right).
\end{equation}
More precisely, if $\vec{Z}\left(t\right)$ has the same direction
as $\vec{Y}\left(t\right)$, then \eqref{eq:freeesti12} holds. For
the general case, \eqref{eq:freeesti11} remains valid.
\end{thm}

\begin{rem}
If $\vec{Y}(t)=\vec{Z}(t)$, one can obtain
\begin{equation}
\left\Vert u^{S}(x,t)\right\Vert _{L_{x}^{\infty}L_{t}^{2}}\lesssim\|f\|_{L^{2}}+\|g\|_{\dot{H}^{1}}+\left\Vert F^{S}\right\Vert _{L_{x}^{\frac{3}{2},1}L_{t}^{2}}.
\end{equation}
The other extreme exponents are $L^{\infty}$ for $t$ and $L^{6}$
for $x$. To be more precise, we have the following endpoint estimates.
\end{rem}

\begin{thm}
\label{thm:reversedlocalF}Let $\vec{Y}(t)$ be an admissible trajectory.
First of all, for the standard case, one has
\begin{equation}
\left\Vert u\right\Vert _{L_{x}^{6,2}L_{t}^{\infty}}\lesssim\|f\|_{L^{2}}+\|g\|_{\dot{H}^{1}}+\left\Vert F\right\Vert _{L_{x}^{\frac{6}{5},2}L_{t}^{\infty}}.\label{eq:localrevS-2}
\end{equation}
For the estimates along the trajectory $\vec{Y}(t)$, one has
\begin{equation}
\left\Vert u^{S}(x,t)\right\Vert _{L_{x}^{6,2}L_{t}^{\infty}}\lesssim\|f\|_{L^{2}}+\|g\|_{\dot{H}^{1}}+\left\Vert \nabla F\right\Vert _{L_{x}^{\frac{6}{5}}L_{t}^{1}}.\label{eq:localrevSS-2}
\end{equation}
Let $\vec{Z}(t)$ be another admissible trajectory, we have the same
estimate as \eqref{eq:localrevSS-2} with $F$ replaced by
\begin{equation}
F^{S'}(x,t):=F\left(x+\vec{Z}(t),t\right).
\end{equation}
\end{thm}

We can extend the estimates above to wave equations with perturbed
Hamiltonian,
\begin{equation}
H=-\Delta+V
\end{equation}
by the structure formulas for wave operators developed in Beceanu-Schlag
\cite{Bec1,BeSch}.
\begin{defn}
\label{def:potenialWO}To ensure the structure of wave operators,
we consider the potential $V$ such that
\begin{equation}
V\in B^{1+}\cap L^{2}\left(\mathbb{R}^{3}\right),
\end{equation}
where
\begin{equation}
B^{\beta}=\left\{ V\,|\sum_{k\in\mathbb{Z}}2^{\beta k}\left\Vert \chi_{\left\{ \left|x\right|\in\left[2^{k},2^{k+1}\right]\right\} }(x)V(x)\right\Vert _{L^{2}}<\infty\right\} .
\end{equation}
and $0$ energy is regular for $H=-\Delta+V$ in the sense that
\[
f=-R_{0}\left(0\right)Vf
\]
has no solution $f\in L^{\infty},\,f\neq0$ where $R_{0}\left(0\right)$
is the free resolvent at $0$. See Beceanu-Schlag \cite{BeSch} for
more detailed discussions.
\end{defn}

\begin{thm}
\label{thm:reversedP}Let $\vec{Y}(t)$ be an admissible trajectory.
Suppose
\begin{equation}
H=-\Delta+V
\end{equation}
satisfies the conditions in Definition \ref{def:potenialWO}. Set
\begin{equation}
u(x,t)=\frac{\sin\left(t\sqrt{H}\right)}{\sqrt{H}}P_{c}f+\cos\left(t\sqrt{H}\right)P_{c}g+\int_{0}^{t}\frac{\sin\left(\left(t-s\right)\sqrt{H}\right)}{\sqrt{H}}P_{c}F(s)\,ds
\end{equation}
and
\begin{equation}
u^{S}(x,t):=u\left(x+\vec{Y}(t),t\right),
\end{equation}
where $P_{c}$ is the projection onto the continuous spectrum of $H$.

Then
\begin{equation}
\left\Vert u^{S}(x,t)\right\Vert _{L_{x}^{\infty}L_{t}^{2}}\lesssim\|f\|_{L^{2}}+\|g\|_{\dot{H}^{1}}+\left\Vert \nabla F\right\Vert _{L_{x}^{1}L_{t}^{2}}.
\end{equation}
If $\vec{Y}(t)$ does not change the direction, then
\begin{equation}
\left\Vert u^{S}(x,t)\right\Vert _{L_{x}^{\infty}L_{t}^{2}}\lesssim\|f\|_{L^{2}}+\|g\|_{\dot{H}^{1}}+\left\Vert F\right\Vert _{L_{d}^{1}L_{\widehat{d}}^{2,1}L_{t}^{2}},
\end{equation}
where $d$ is the direction of $\vec{Y}(t)$ and $\hat{d}$ is the
subspace orthogonal to the direction $d$.

Let $\vec{Z}(t)$ be another admissible trajectory, we have the same
estimates as above with $F$ replaced by
\begin{equation}
F^{S'}(x,t):=F\left(x+\vec{Z}(t),t\right)
\end{equation}
similar to the free case in Theorem \ref{thm:reversedF}.
\end{thm}

We also have the perturbed version of the second endpoint reversed
space-time estimates.
\begin{thm}
\label{thm:reveredLocalP}Let $\vec{Y}(t)$ be an admissible trajectory.
Suppose
\begin{equation}
H=-\Delta+V
\end{equation}
satisfies the conditions in Definition \ref{def:potenialWO}. Set
\begin{equation}
u(x,t)=\frac{\sin\left(t\sqrt{H}\right)}{\sqrt{H}}P_{c}f+\cos\left(t\sqrt{H}\right)P_{c}g+\int_{0}^{t}\frac{\sin\left(\left(t-s\right)\sqrt{H}\right)}{\sqrt{H}}P_{c}F(s)\,ds
\end{equation}
and
\begin{equation}
u^{S}(x,t):=u\left(x+\vec{Y}(t),t\right),
\end{equation}
where $P_{c}$ is the projection onto the continuous spectrum of $H$.
First of all, for the standard case, one has
\begin{equation}
\left\Vert u\right\Vert _{L_{x}^{6,2}L_{t}^{\infty}}\lesssim\|f\|_{L^{2}}+\|g\|_{\dot{H}^{1}}+\left\Vert F\right\Vert _{L_{x}^{\frac{6}{5},2}L_{t}^{\infty}}.\label{eq:localrevS-1-1}
\end{equation}
Consider the estimates along the trajectory $\vec{Y}(t)$, one has
\begin{equation}
\left\Vert u^{S}(x,t)\right\Vert _{L_{x}^{6,2}L_{t}^{\infty}}\lesssim\|f\|_{L^{2}}+\|g\|_{\dot{H}^{1}}+\left\Vert \nabla F\right\Vert _{L_{x}^{\frac{6}{5}}L_{t}^{1}}.\label{eq:localrevSS-1-1}
\end{equation}
Let $\vec{Z}(t)$ be another admissible trajectory, we have the same
estimates as above with $F$ replaced by
\begin{equation}
F^{S'}(x,t):=F\left(x+\vec{Z}(t),t\right)
\end{equation}
similar to the free case in Theorem \ref{thm:reversedlocalF}.
\end{thm}

We will rely on the structure formula of the wave operators by Beceanu-Schlag
\cite{BeSch}. Although one can obtain similar results without using
the structure formula, see \cite{GC2}, the goal of our exposition
is the illustrate a general strategy that one can pass the estimates
for the free evolution to the perturbed one via the structure formula
provided there are some symmetries of the free solution operators.

\smallskip

As applications of the estimates above, we study both regular and
reversed Strichartz estimates for scattering states to a wave equation
with a moving potential with the trajectory asymptotically like a
straight line. Suppose $\vec{Y}(t)\in\mathbb{R}^{3}$ is a trajectory
such that there exists $\vec{\mu}\in\mathbb{R}^{3},\,\left|\vec{\mu}\right|<1$
with
\begin{equation}
\left|\vec{Y}(t)-\vec{\mu}t\right|\lesssim\left\langle t\right\rangle ^{-\beta},\,\beta>1.\label{eq:trajcond}
\end{equation}
Consider
\begin{equation}
\partial_{tt}u-\Delta u+V\left(x-\vec{Y}(t)\right)u=0\label{eq:meq}
\end{equation}
with initial data
\begin{equation}
u(x,0)=g(x),\,u_{t}(x,0)=f(x).
\end{equation}
\begin{rem}
Actually, with much more complicated and technical analysis, one can
replace the decay rate with much weaker condition using the idea from
Beceanu \cite{Bec3} and Nakanishi-Schlag \cite{NS2} adapted to the
wave equation. But for simplicity, we assume this decay rate as some
analysis in Rodnianski-Schlag-Soffer \cite{RSS2}.
\end{rem}

An indispensable tool we need to study wave equations with moving
potentials is the Lorentz transformations. From the setting above,
without loss of generality, we assume $\vec{\mu}$ is along $\overrightarrow{e_{1}}$.
We apply the Lorentz transformation $L$ with respect to a moving
frame with speed $\left|\mu\right|<1$ along the $x_{1}$ direction.
Writing down the Lorentz transformation explicitly, we have
\begin{equation}
\begin{cases}
t'=\gamma\left(t-\mu x_{1}\right)\\
x_{1}'=\gamma\left(x_{1}-\mu t\right)\\
x_{2}'=x_{2}\\
x_{3}'=x_{3}
\end{cases}\label{eq:LorentzT}
\end{equation}
with
\begin{equation}
\gamma=\frac{1}{\sqrt{1-\left|\mu\right|^{2}}}.
\end{equation}
We can also write down the inverse transformation of the one above:
\begin{equation}
\begin{cases}
t=\gamma\left(t'+vx_{1}'\right)\\
x_{1}=\gamma\left(x_{1}'+\mu t'\right)\\
x_{2}=x_{2}'\\
x_{3}=x_{3}'
\end{cases}.\label{eq:InvLorentT}
\end{equation}
Under the Lorentz transformation $L$, if we use the subscript $L$
to denote a function with respect to the new coordinate $\left(x',t'\right)$,
we have
\begin{equation}
u_{L}\left(x_{1}',x_{2}',x_{3}',t'\right)=u\left(\gamma\left(x_{1}'+\mu t'\right),x_{2}',x_{3}',\gamma\left(t'+\mu x_{1}'\right)\right)\label{eq:Lcoordinate}
\end{equation}
and
\begin{equation}
u(x,t)=u_{L}\left(\gamma\left(x_{1}-\mu t\right),x_{2},x_{3},\gamma\left(t-v\mu x\right)\right).\label{eq:ILcoordinate}
\end{equation}
In order to study the equation with time-dependent potentials, we
need to introduce a suitable projection. Given $\vec{\mu}$ as above,
we denote
\[
H=-\Delta+V\left(\sqrt{1-\left|\mu\right|^{2}}x_{1},x_{2},x_{3}\right).
\]
\begin{defn}
\label{def:potenialWO-1}To consider the moving potential problem,
for given $\vec{\mu},$we assume that
\[
\left|V\right|\lesssim\frac{1}{\left\langle x\right\rangle ^{\alpha}},\ \alpha>3
\]
and there is no zero eigenfunctions nor resonances for
\[
H=-\Delta+V\left(\sqrt{1-\left|\mu\right|^{2}}x_{1},x_{2},x_{3}\right).
\]
Recall that $\psi$ is a resonance at $0$ if it is a distributional
solution of the equation $H\psi=0$ which belongs to the space $L^{2}\left(\left\langle x\right\rangle ^{-\sigma}dx\right):=\left\{ f:\,\left\langle x\right\rangle ^{-\sigma}f\in L^{2}\right\} $
for any $\sigma>\frac{1}{2}$, but not for $\sigma=\frac{1}{2}.$ 
\end{defn}

\begin{rem}
Here, we impose the spectral conditions for the Schr\"odinger operator
with respect to a fixed $\vec{\mu}$. Alternatively, we can impose
the spectral conditions on $-\Delta+V\left(x\right)$ and the consider
the Lorent boost of $V$. For this setting, see for example \cite{CJ}
where the potential is given by the Lorentz boost of a soliton.
\end{rem}

Let $m_{1},\,\ldots,\,m_{w}$ be the normalized bound states of $H$
associated to the negative eigenvalues $-\lambda_{1}^{2},\,\ldots,\,-\lambda_{w}^{2}$
respectively (notice that by our assumptions, $0$ is not an eigenvalue).
In other words, we assume that
\begin{equation}
Hm_{i}=-\lambda_{i}^{2}m_{i},\,\,\,m_{i}\in L^{2},\,\lambda_{i}>0.
\end{equation}
We denote by $P_{b}$ the projections on the the bound states of $H$
and let $P_{c}=Id-P_{b}$. To be more explicit, we have
\begin{equation}
P_{b}=\sum_{j=1}^{\ell}\left\langle \cdot,m_{j}\right\rangle m_{j}.
\end{equation}
With Lorentz transformations $L$ associated to the moving frame $\left(x-\vec{\mu}t,t\right)$,
we use the subscript $L$ to denote a function under the new frame
$\left(x',t'\right)$.
\begin{defn}[Asymptotic orthogonality]
\label{AO}Let $u$ solve
\begin{equation}
\partial_{tt}u-\Delta u+V\left(x-\vec{Y}(t)\right)u=0,\label{eq:eqBSsec-1}
\end{equation}
\[
u(x,0)=g(x),\,u_{t}(x,0)=f(x),
\]
where the potential and the trajectory satisfy Definition \ref{def:potenialWO-1}
and condition \eqref{eq:trajcond} respectively. If $u$ also satisfies
\begin{equation}
\left\Vert P_{b}u_{L}(t')\right\Vert _{L_{x'}^{2}}\rightarrow0\,\,\,t,t'\rightarrow\infty,\label{eq:ao2-1}
\end{equation}
we call it a scattering state. 
\end{defn}

\begin{rem}
The existence of the scattering state here can be understood as the
subspace which generated sub-exponential growth in the setting of
exponential dichotomies, see \cite{CJ}. In particular, in \cite{CJ},
this subspace is proved to be of co-dimension $w$ using the notations
 above.
\end{rem}

\begin{thm}[Strichartz estimates]
\label{thm:Stri}Suppose $u$ is a scattering
state in the sense of Definition \ref{AO} which solves the equation
\eqref{eq:meq}.\textup{ }Then for $p>2$ and $(p,q)$ satisfying
\begin{equation}
\frac{1}{2}=\frac{1}{p}+\frac{3}{q},
\end{equation}
we have
\begin{equation}
\|u\|_{L_{t}^{p}\left([0,\infty),\,L_{x}^{q}\right)}\lesssim\|f\|_{L^{2}}+\|g\|_{\dot{H}^{1}}.
\end{equation}
\end{thm}

The theorem above can be extended to the inhomogeneous case, see for
example \cite{GC2}.

Secondly, one has the energy estimate:
\begin{thm}[Energy estimate]
\label{thm:Energy}Suppose $u$ is a scattering
state in the sense of Definition \ref{AO} which solves the equation
\eqref{eq:meq}.\textup{ }Then we have
\begin{equation}
\sup_{t\geq0}\left(\|\nabla u(t)\|_{L^{2}}+\|u_{t}(t)\|_{L^{2}}\right)\lesssim\|f\|_{L^{2}}+\|g\|_{\dot{H}^{1}}.\label{eq:StriCharWOB-1}
\end{equation}
\end{thm}

We also obtain the endpoint reversed $L_{x}^{\infty}L_{t}^{2}$ Strichartz
estimates for $u$. 
\begin{thm}[Endpoint reversed Strichartz estimate]
\label{thm:EndRStri}Let $\vec{Z}(t)$
be an admissible trajectory. Suppose $u$ is a scattering state in
the sense of Definition \ref{AO} which solves the equation \eqref{eq:meq}.\textup{
}Then
\begin{equation}
\sup_{x\in\mathbb{R}^{3}}\int_{0}^{\infty}\left|u(x,t)\right|^{2}dt\lesssim\left(\|f\|_{L^{2}}+\|g\|_{\dot{H}^{1}}\right)^{2},
\end{equation}
and
\begin{equation}
\sup_{x\in\mathbb{R}^{3}}\int_{0}^{\infty}\left|u(x+\vec{Z}(t),t)\right|^{2}dt\lesssim\left(\|f\|_{L^{2}}+\|g\|_{\dot{H}^{1}}\right)^{2}.
\end{equation}
\end{thm}

With the endpoint estimate along $\left(x+\vec{Y}(t),t\right)$, one
can derive the boundedness of the total energy. We denote the total
energy of the system as
\begin{equation}
E_{V}(t)=\int\left|\nabla_{x}u\right|^{2}+\left|\partial_{t}u\right|^{2}+V\left(x-\vec{Y}(t)\right)\left|u\right|^{2}dx.
\end{equation}
\begin{cor}[Boundedness of the total energy]
\label{cor:ene}Suppose $u$ is
a scattering state in the sense of Definition \ref{AO} which solves
the equation \eqref{eq:meq}. Assume
\begin{equation}
\left\Vert \nabla V\right\Vert _{L^{1}}<\infty,
\end{equation}
then $E_{V}(t)$ is bounded by the initial energy independently of
$t$,
\begin{equation}
\sup_{t\geq0}\left|E_{V}(t)\right|\lesssim\left\Vert \left(g,f\right)\right\Vert _{\dot{H}^{1}\times L^{2}}^{2}.
\end{equation}
\end{cor}

\subsection*{Notation}

\textquotedblleft $A:=B\lyxmathsym{\textquotedblright}$ or $\lyxmathsym{\textquotedblleft}B=:A\lyxmathsym{\textquotedblright}$
is the definition of $A$ by means of the expression $B$. We use
the notation $\langle x\rangle=\left(1+|x|^{2}\right)^{\frac{1}{2}}$.
The bracket $\left\langle \cdot,\cdot\right\rangle $ denotes the
distributional pairing and the scalar product in the spaces $L^{2}$,
$L^{2}\times L^{2}$ . For positive quantities $a$ and $b$, we write
$a\lesssim b$ for $a\leq Cb$ where $C$ is some prescribed constant.
Also $a\simeq b$ for $a\lesssim b$ and $b\lesssim a$. Throughout,
we use $\partial_{tt}u:=\frac{\partial^{2}}{\partial t\partial t}$,
$u_{t}:=\frac{\partial}{\partial_{t}}u$, $\Delta:=\sum_{i=1}^{n}\frac{\partial^{2}}{\partial x_{i}\partial x_{i}}$
and occasionally, $\square:=-\partial_{tt}+\Delta$. 

\subsection*{Organization }

The paper is organized as follows: In Section \ref{sec:Prelim}, we
discuss some preliminary results for the free wave equation and the
wave equation with a stationary potential. In Section \ref{sec: Lorentz},
we will analyze the change of the energy under Lorentz transformations.
Agmon's estimates are also presented as a consequence of our comparison
results. In Section \ref{sec:revered}, the endpoint reversed Strichartz
estimates of homogeneous and inhomogeneous forms are derived along
admissible trajectories. In Section \ref{sec:one}, we show Strichartz
estimates, energy estimates, the local energy decay and the boundedness
of the total energy for a scattering state to the wave equation with
a moving potential. Finally, in Section \ref{sec:Scattering}, we
confirm that a scattering state indeed scatters to a solution to the
free wave equation and also obtain a version of the asymptotic completeness
description of the wave equations with one moving potential. In appendices,
for the sake of completeness, we show the dispersive estimates for
wave equations in $\mathbb{R}^{3}$ based on the idea of reversed
Strichartz estimates, the local energy decay of free wave equations
and the global existence of solutions to the wave equation with a
time-dependent potential. A Fourier analytic proof of the endpoint
reversed Strichartz estimates is also presented.

\subsection*{Acknowledgment}

I want to thank Marius Beceanu for many useful discussions.

\section{Preliminaries\label{sec:Prelim}}

\subsection{Strichartz estimates and the endpoint reversed Strichartz estimates}

We start with Strichartz estimates for free wave equations. Strichartz
estimates can be derived abstractly from these dispersive inequalities
and the energy inequality. The following theorem is standard. To be
consistent with our later discussion, we only state the energy level
estimates. One can find full details with all possible estimates and
proofs in, for example, Keel-Tao \cite{KT}.
\begin{thm}[Strichartz estimates]
\label{thm:StrichF}Suppose
\begin{equation}
\partial_{tt}u-\Delta u=F
\end{equation}
with initial data
\begin{equation}
u(x,0)=g(x),\,u_{t}(x,0)=f(x).
\end{equation}
Then for $p,\,a>2$, $\left(p,q\right),\,\left(a,b\right)$ satisfying
\begin{equation}
\frac{1}{2}=\frac{1}{p}+\frac{3}{q}=\frac{1}{a}+\frac{3}{b}
\end{equation}
we have
\begin{equation}
\|u\|_{L_{t}^{p}L_{x}^{q}}\lesssim\|g\|_{\dot{H}^{1}}+\|f\|_{L^{2}}+\left\Vert F\right\Vert _{L_{t}^{a'}L_{x}^{b'}}\label{eq:StrichF}
\end{equation}
where $\frac{1}{a}+\frac{1}{a'}=1,\,\frac{1}{b}+\frac{1}{b'}=1.$
\end{thm}

The endpoint $\left(p,q\right)=\left(2,\infty\right)$ can be recovered
for radial functions in Klainerman-Machedon \cite{KM} for the homogeneous
case and Jia-Liu-Schlag-Xu \cite{JLSX} for the inhomogeneous case.
The endpoint estimate can also be obtained when a small amount of
smoothing (either in the Sobolev sense, or in relaxing the integrability)
is applied to the angular variable, see Machihara-Nakamura-Nakanishi-Ozawa
\cite{MNNO}. 
\begin{thm}[\cite{MNNO}]
\label{thm:inhomAR}  For any $1\leq p<\infty$, suppose
$u$ solves the free wave equation
\begin{equation}
\partial_{tt}u-\Delta u=0
\end{equation}
 with initial data
\begin{equation}
u(x,0)=g(x),\,u_{t}(x,0)=f(x).
\end{equation}
Then
\begin{equation}
\|u\|_{L_{t}^{2}L_{r}^{\infty}L_{\omega}^{p}}\le C(p)\left(\|f\|_{L^{2}}+\|g\|_{\dot{H}^{1}}\right).\label{eq:inhomoAR}
\end{equation}
\end{thm}

The regular Strichartz estimates fail at the endpoint. But if one
switches the order of space-time integration, it is possible to estimate
the solution using the fact that the solution decays quickly away
from the light cone. Therefore, we introduce reversed Strichartz estimates.
Since only the endpoint reversed Stricharz estimate will be used later
on, we will restrict our to that case. 
\begin{thm}[Endpoint reversed Strichartz estimate]
\label{thm:EndRStrichF}Suppose
\begin{equation}
\partial_{tt}u-\Delta u=F
\end{equation}
with initial data
\begin{equation}
u(x,0)=g(x),\,u_{t}(x,0)=f(x).
\end{equation}
Then
\begin{equation}
\left\Vert u\right\Vert _{L_{x}^{\infty}L_{t}^{2}}\lesssim\|f\|_{L^{2}}+\|g\|_{\dot{H}^{1}}+\left\Vert F\right\Vert _{L_{x}^{\frac{3}{2},1}L_{t}^{2}}.\label{eq:EndRStrichF}
\end{equation}
and
\begin{equation}
\left\Vert u\right\Vert _{L_{x}^{6,2}L_{t}^{\infty}}\lesssim\|f\|_{L^{2}}+\|g\|_{\dot{H}^{1}}+\left\Vert F\right\Vert _{L_{x}^{\frac{6}{5},2}L_{t}^{\infty}.}
\end{equation}
\end{thm}

See Section \ref{sec:revered} for the detailed proof. For \eqref{eq:EndRStrichF},
one can find an alternative proof based on the Fourier transform in
Appendix D.

The results above from Theorem \ref{thm:StrichF} and Theorem \ref{thm:EndRStrichF}
can be generalized to the wave equation with a real stationary potential. 

For the perturbed Hamiltonian,
\begin{equation}
H=-\Delta+V
\end{equation}
satisfy Definition \ref{def:potenialWO}, we consider the wave equation
with potential in $\mathbb{R}^{3}$:

\begin{equation}
\partial_{tt}u-\Delta u+Vu=0
\end{equation}
with initial data
\begin{equation}
u(x,0)=g(x),\,u_{t}(x,0)=f(x).
\end{equation}
One can write down the solution to it explicitly:
\begin{equation}
u=\frac{\sin\left(t\sqrt{H}\right)}{\sqrt{H}}f+\cos\left(t\sqrt{H}\right)g.
\end{equation}
Let $P_{b}$ be the projection onto the point spectrum of $H$, $P_{c}=I-P_{b}$
be the projection onto the continuous spectrum of $H$. 
\begin{rem}
We do not try to get the most optimal regularity and decay conditions
on the potential. For the optimal cases, one can check the conditions
in \cite{BecGo}
\end{rem}

With the setting above, we formulate the results from Beceanu-Goldberg
\cite{BecGo}.
\begin{thm}[Strichartz and reversed Strichartz estimates]
\label{thm:PStriRStrich}Consider
the perturbed Hamiltonian $H=-\Delta+V$ in $\mathbb{R}^{3}$ as above.
Then for all $p>2$, and $\left(p,q\right)$ satisfying
\begin{equation}
\frac{1}{2}=\frac{1}{p}+\frac{3}{q}
\end{equation}
we have
\begin{equation}
\left\Vert \frac{\sin\left(t\sqrt{H}\right)}{\sqrt{H}}P_{c}f+\cos\left(t\sqrt{H}\right)P_{c}g\right\Vert _{L_{t}^{p}L_{x}^{q}}\lesssim\|g\|_{\dot{H}^{1}}+\|f\|_{L^{2}}.\label{PSrich}
\end{equation}
For the endpoint of reversed Strichartz estimates, we have
\begin{equation}
\left\Vert \frac{\sin\left(t\sqrt{H}\right)}{\sqrt{H}}P_{c}f+\cos\left(t\sqrt{H}\right)P_{c}g\right\Vert _{L_{x}^{\infty}L_{t}^{2}}\lesssim\|f\|_{L^{2}}+\|g\|_{\dot{H}^{1}},\label{eq:PEndRSch}
\end{equation}
\begin{equation}
\left\Vert \int_{0}^{t}\frac{\sin\left((t-s)\sqrt{H}\right)}{\sqrt{H}}P_{c}F(s)\,ds\right\Vert _{L_{x}^{\infty}L_{t}^{2}}\lesssim\left\Vert F\right\Vert _{L_{x}^{\frac{3}{2},1}L_{t}^{2}}.\label{eq:PEndRSIn}
\end{equation}
One also has
\begin{equation}
\left\Vert \frac{\sin\left(t\sqrt{H}\right)}{\sqrt{H}}P_{c}f+\cos\left(t\sqrt{H}\right)P_{c}g\right\Vert _{L_{x}^{6,2}L_{t}^{\infty}}\lesssim\|f\|_{L^{2}}+\|g\|_{\dot{H}^{1}},
\end{equation}
\begin{equation}
\left\Vert \int_{0}^{t}\frac{\sin\left((t-s)\sqrt{H}\right)}{\sqrt{H}}P_{c}F(s)\,ds\right\Vert _{L_{x}^{6,2}L_{t}^{2}}\lesssim\left\Vert F\right\Vert _{L_{x}^{\frac{6}{5},2}L_{t}^{\infty}}.
\end{equation}
\end{thm}

One can find detailed arguments and more estimates in Beceanu-Goldberg
\cite{BecGo}. We will apply the structure of the wave operators to
show the theorem above in Section \ref{sec:revered}.

\subsection{Structure of wave operators and its applications}

Next, we discuss the structure of wave operators. Again consider
\begin{equation}
H=-\Delta+V.
\end{equation}
For wave operators, we define
\begin{equation}
W^{+}=s-\lim_{t\rightarrow\infty}e^{itH}e^{it\Delta}.
\end{equation}
We know
\begin{equation}
W^{+}\left(-\Delta\right)=HW^{+}
\end{equation}
and
\begin{equation}
\left(W^{+}\right)^{*}=s-\lim_{t\rightarrow\infty}e^{itH_{0}}e^{-itH}P_{c}.
\end{equation}
By Beceanu and Schlag \cite{BeSch}, we have the following structure
formula for $W^{+}$ and $\left(W^{+}\right)^{*}$. 
\begin{thm}[\cite{BeSch}]
\label{thm:structure}Assume
\[
H=-\Delta+V
\]
satisfies Definition \ref{def:potenialWO}. Then for both $W^{+}$
and $\left(W^{+}\right)^{*},$ we have for $f\in L^{2}$,
\begin{equation}
Wf(x)=f(x)+\int_{\mathbb{S}^{2}}\int_{\mathbb{R}^{3}}g(x,y,\omega)f\left(S_{\omega}x+y\right)\,dyd\omega,\label{eq:lwaveop}
\end{equation}
for some $g(x,y,\omega)$ such that
\begin{equation}
\int_{\mathbb{S}^{2}}\int_{\mathbb{R}^{3}}\left\Vert g(x,y,\omega)\right\Vert _{L_{x}^{\infty}}dyd\omega<\infty
\end{equation}
and where
\begin{equation}
S_{\omega}x=x-(2x\cdot\omega)\omega.
\end{equation}
is the reflection by the plane orthogonal to $\omega$. Here $W$
is either of $W^{+}$ or $\left(W^{+}\right)^{*}$.
\end{thm}

The structure formula \eqref{eq:lwaveop} in Theorem \ref{thm:structure}
is useful to obtain estimates for the perturbed operators. One can
easily pass many estimates from the free case to the perturbed case
provided the solution operators of the free problem commute with certain
symmetries. Here we illustrate this idea by a concrete computation
based on Theorem \ref{thm:inhomAR}. 
\begin{thm}
\label{thm:inhomARH} Assume $H=-\Delta+V$ satisfies Definition \ref{def:potenialWO}.
Setting
\begin{equation}
u^{H}=\frac{\sin\left(t\sqrt{H}\right)}{\sqrt{H}}P_{c}f+\cos\left(t\sqrt{H}\right)P_{c}g,
\end{equation}
then for any $1\leq p<\infty$, one has
\begin{equation}
\left\Vert u^{H}\right\Vert {}_{L_{t}^{2}L_{r}^{\infty}L_{\omega}^{p}}\le C(p)\left(\|f\|_{L^{2}}+\|g\|_{\dot{H}^{1}}\right)\label{eq:inhomoARH}
\end{equation}
\end{thm}

\begin{proof}
It suffices to consider
\begin{equation}
u^{H}=\frac{\sin\left(t\sqrt{H}\right)}{\sqrt{H}}P_{c}f.
\end{equation}
By construction,
\begin{equation}
\frac{\sin\left(t\sqrt{H}\right)}{\sqrt{H}}P_{c}=W^{+}\frac{\sin\left(t\sqrt{-\Delta}\right)}{\sqrt{-\Delta}}\left(W^{+}\right)^{*}.
\end{equation}
Hence
\begin{equation}
\frac{\sin\left(t\sqrt{H}\right)}{\sqrt{H}}P_{c}f=W^{+}\frac{\sin\left(t\sqrt{-\Delta}\right)}{\sqrt{-\Delta}}\left(W^{+}\right)^{*}P_{c}f.
\end{equation}
Denoting
\begin{equation}
h=\left(W^{+}\right)^{*}P_{c}f,
\end{equation}
we have
\begin{equation}
\|P_{c}f\|_{L^{2}}\simeq\|h\|_{L^{2}}.
\end{equation}
Setting
\begin{equation}
G=\frac{\sin\left(t\sqrt{-\Delta}\right)}{\sqrt{-\Delta}}h,
\end{equation}
by Theorem \ref{thm:structure}, it is sufficient to consider the
boundedness of
\begin{equation}
G+\int_{\mathbb{S}^{2}}\int_{\mathbb{R}^{3}}g(x,y,\tau)G\left(S_{\tau}x+y\right)\,dyd\tau.
\end{equation}
Clearly, by Theorem \ref{thm:inhomAR},
\begin{equation}
\|G\|_{L_{t}^{2}L_{r}^{\infty}L_{\omega}^{p}}\lesssim\|h\|_{L^{2}}\simeq\|P_{c}f\|_{L^{2}}.
\end{equation}
Next, by Minkowski's inequality,
\begin{eqnarray}
\left\Vert \int_{\mathbb{S}^{2}}\int_{\mathbb{R}^{3}}g(x,y,\tau)G\left(S_{\tau}x+y\right)\,dyd\tau\right\Vert _{L_{t}^{2}L_{r}^{\infty}L_{\omega}^{p}}\\
\lesssim\int_{\mathbb{S}^{2}}\int_{\mathbb{R}^{3}}\left\Vert g(x,y,\tau)G\left(S_{\tau}x+y\right)\right\Vert {}_{L_{t}^{2}L_{r}^{\infty}L_{\omega}^{p}}dyd\tau\nonumber 
\end{eqnarray}
and
\begin{equation}
\left\Vert g(x,y,\tau)G\left(S_{\tau}x+y\right)\right\Vert {}_{L_{t}^{2}L_{r}^{\infty}L_{\omega}^{p}}\lesssim\left\Vert g(x,y,\tau)\right\Vert _{L_{x}^{\infty}}\left\Vert G\left(S_{\tau}x+y\right)\right\Vert {}_{L_{t}^{2}L_{r}^{\infty}L_{\omega}^{p}}.
\end{equation}
Since reflections with respect to a fixed plane and translations commute
with the solution of a free wave equation, we obtain
\begin{equation}
G\left(S_{\tau}x+y\right)=\frac{\sin\left(t\sqrt{-\Delta}\right)}{\sqrt{-\Delta}}h\left(S_{\tau}x+y\right).
\end{equation}
Therefore,
\begin{equation}
\left\Vert G\left(S_{\tau}x+y\right)\right\Vert {}_{L_{t}^{2}L_{r}^{\infty}L_{\omega}^{p}}\lesssim\|h\left(S_{\tau}x+y\right)\|_{L^{2}}\lesssim\|h\|_{L^{2}}\simeq\|P_{c}f\|_{L^{2}}.
\end{equation}
It follows
\begin{eqnarray}
\left\Vert G+\int_{\mathbb{S}^{2}}\int_{\mathbb{R}^{3}}g(x,y,\tau)G\left(S_{\tau}x+y\right)\,dyd\tau\right\Vert {}_{L_{t}^{2}L_{r}^{\infty}L_{\omega}^{p}}\\
\lesssim\left(1+\int_{\mathbb{S}^{2}}\int_{\mathbb{R}^{3}}\left\Vert g(x,y,\tau)\right\Vert _{L_{x}^{\infty}}dyd\tau\right)\|P_{c}f\|_{L^{2}}\lesssim\|f\|_{L^{2}}.\nonumber 
\end{eqnarray}
Then we conclude
\begin{equation}
\left\Vert u^{H}\right\Vert {}_{L_{t}^{2}L_{r}^{\infty}L_{\omega}^{p}}\le C(p,V)\left(\|f\|_{L^{2}}+\|g\|_{\dot{H}^{1}}\right),
\end{equation}
as claimed.
\end{proof}
One can do similar arguments to obtain many other estimates for the
perturbed wave equations, for example the local energy decay estimate,
the energy estimate and many weighted estimates.

The following Christ-Kiselev Lemma is important in our derivation
of Strichartz estimates.
\begin{lem}[Christ-Kiselev]
\label{lem:Christ-Kiselev}Let $X$, $Y$ be two
Banach spaces and let $T$ be a bounded linear operator from $L^{\beta}\left(\mathbb{R}^{+};X\right)$
to $L^{\gamma}\left(\mathbb{R}^{+};Y\right)$, such that
\begin{equation}
Tf(t)=\int_{0}^{\infty}K(t,s)f(s)\,ds.
\end{equation}
Then the operator
\begin{equation}
\widetilde{T}f=\int_{0}^{t}K(t,s)f(s)\,ds
\end{equation}
 is bounded from $L^{\beta}\left(\mathbb{R}^{+};X\right)$ to $L^{\gamma}\left(\mathbb{R}^{+};Y\right)$
provided $\beta<\gamma$, and
\begin{equation}
\left\Vert \widetilde{T}\right\Vert \leq C(\beta,\gamma)\left\Vert T\right\Vert 
\end{equation}
with
\begin{equation}
C(\beta,\gamma)=\left(1-2^{\frac{1}{\gamma}-\frac{1}{\beta}}\right)^{-1}.
\end{equation}
\end{lem}

\section{Lorentz Transformations and Energy\label{sec: Lorentz}}

When we consider wave equations with moving potentials, Lorentz transformations
will be important for us to reduce some estimates to stationary cases.
In order to approach our problem from the viewpoint of Lorentz transformations
as in \cite{GC2}, the first natural step is to understand the change
of energy under Lorentz transformations. 

Indeed, in \cite{GC2}, we shown that under Lorentz transformations,
the energy stays comparable to that of the initial data. The method
in \cite{GC2} is based on integration by parts. Here we present an
alternative approach based on the local energy conservation which
is more natural and requires less decay of the potential. We notice
that the method in \cite{GC2} can be viewed as the differential version
of the argument here.

Throughout this section, we perform a Lorentz transformation with
respect to a moving frame with speed $\left|v\right|<1$, say, along
the $x_{1}$ direction, i.e., the velocity is
\begin{equation}
\vec{v}=\left(v,0,0\right).\label{eq:l1}
\end{equation}
Recall that after applying the Lorentz transformation, for function
$u$, under the new coordinates, we denote
\begin{equation}
u_{L}\left(x_{1}',x_{2}',x_{3}',t'\right)=u\left(\gamma\left(x_{1}'+vt'\right),x_{2}',x_{3}',\gamma\left(t'+vx_{1}'\right)\right).\label{eq:l6}
\end{equation}
Now let $u$ be a solution to some wave equation and set $t'=0$.
We notice that in order to show under Lorentz transformations, the
energy stays comparable to that of the initial data, up to an absolute
constant it suffices to prove
\begin{eqnarray}
\int\left|\nabla_{x}u\left(x_{1},x_{2},x_{3},vx_{1}\right)\right|^{2}+\left|\partial_{t}u\left(x_{1},x_{2},x_{3},vx_{1}\right)\right|^{2}dx\nonumber \\
\simeq\int\left|\nabla_{x}u\left(x_{1},x_{2},x_{3},0\right)\right|^{2}+\left|\partial_{t}u\left(x_{1},x_{2},x_{3},0\right)\right|^{2}dx
\end{eqnarray}
where the implicit constant depends on $\left|\vec{v}\right|$.

Throughout this section, we will assume all functions are smooth and
decay fast. We will obtain estimates independent of the additional
smoothness assumption. It is easy to pass the estimates to general
cases with a density argument. 
\begin{rem}
\label{rem:dim}One can observe that all discussions in this section
hold for $\mathbb{R}^{n}$ for $n\geq3$. 
\end{rem}

\subsection{Energy comparison}

In this section, a more general situation is analyzed. We consider
wave equations with time-dependent potentials
\begin{equation}
\partial_{tt}u-\Delta u+V(x,t)u=0
\end{equation}
with
\begin{equation}
\left|V(x,\mu x_{1})\right|\lesssim\frac{1}{\left\langle x\right\rangle ^{2}}
\end{equation}
 uniformly for $0\leq\left|\mu\right|<\left|v\right|<1$. These in
particular apply to wave equations with moving potentials with speed
strictly less than $1$. For example, if the potential is of the from
\begin{equation}
V(x,t)=V\left(x-\vec{\nu}t\right)
\end{equation}
 with
\begin{equation}
\left|V(x)\right|\lesssim\frac{1}{\left\langle x\right\rangle ^{2}}
\end{equation}
then it is transparent that
\begin{equation}
\left|V(x,\mu x_{1})\right|=\left|V(x-\nu\mu x_{1})\right|\lesssim\frac{1}{\left\langle x\right\rangle ^{2}}.
\end{equation}
\begin{rem}
The condition here is to ensure Hardy's inequality can be applied
with weight $\left|V\right|$. Again here, we do not try to get the
most optimal condition.
\end{rem}

\subsubsection{Homogeneous comparison}

Suppose
\begin{equation}
\partial_{tt}u-\Delta u+V(x,t)u=0,
\end{equation}
then it is clear that
\begin{eqnarray}
0 & = & u_{t}\left(\square u-V(t)u\right)\nonumber \\
 & = & -\partial_{t}\left(\frac{\left|u_{t}\right|^{2}}{2}+\frac{\left|u_{x}\right|^{2}}{2}\right)+\mathrm{div}\left(\nabla uu_{t}\right)-V(x,t)uu_{t}.\label{eq:vector}
\end{eqnarray}
\begin{thm}
\label{thm:generalC}Let $\left|v\right|<1$. Suppose
\begin{equation}
\partial_{tt}u-\Delta u+V(x,t)u=0
\end{equation}
and
\begin{equation}
\left|V(x,\mu x_{1})\right|\lesssim\frac{1}{\left\langle x\right\rangle ^{2}}
\end{equation}
 for $0\leq\left|\mu\right|\leq\left|v\right|<1$. Then

\begin{eqnarray}
\int\left|\nabla_{x}u\left(x_{1},x_{2},x_{3},vx_{1}\right)\right|^{2}+\left|\partial_{t}u\left(x_{1},x_{2},x_{3},vx_{1}\right)\right|^{2}dx\nonumber \\
\simeq\int\left|\nabla_{x}u\left(x_{1},x_{2},x_{3},0\right)\right|^{2}+\left|\partial_{t}u\left(x_{1},x_{2},x_{3},0\right)\right|^{2}dx,\label{eq:generalC}
\end{eqnarray}
where the implicit constant depends on $v$ and $V$.
\end{thm}

\begin{proof}
We apply the space-time divergence theorem to
\begin{equation}
\left(\nabla uu_{t},-\left(\frac{\left|u_{t}\right|^{2}}{2}+\frac{\left|u_{x}\right|^{2}}{2}\right)\right)
\end{equation}
in the region bounded in planes $\left(x_{1},x_{2},x_{3},vx_{1}\right)$
and $\left(x_{1},x_{2},x_{3},0\right)$. Focusing on $x_{1}\geq0$
part which we denote as $Y^{+}$ (the other piece follows similarly),
we note that the unit outward-pointing normal vector on the plane
$Y_{T}:=\left(x_{1},x_{2},x_{3},vx_{1}\right)$ is
\begin{equation}
\frac{1}{\sqrt{v^{2}+1}}\left(-v,0,0,1\right)
\end{equation}
and the outward-pointing normal vector on the bottom of $Y^{+}$,
$Y_{B}$, is
\begin{equation}
(0,0,0,-1).
\end{equation}
Hence
\begin{eqnarray}
\frac{1}{\sqrt{v^{2}+1}}\int_{Y_{T}^{+}}\left[v\partial_{x_{1}}uu_{t}-\left(\frac{\left|u_{t}\right|^{2}}{2}+\frac{\left|u_{x}\right|^{2}}{2}\right)\right]\,dx+\int_{Y_{B}^{+}}\left(\frac{\left|u_{t}\right|^{2}}{2}+\frac{\left|u_{x}\right|^{2}}{2}\right)\,dx\label{eq:div1}\\
=\int_{Y^{+}}V(x,t)uu_{t}\,dxdt.
\end{eqnarray}
We will prove
\begin{eqnarray}
\int\left|\nabla_{x}u\left(x_{1},x_{2},x_{3},vx_{1}\right)\right|^{2}+\left|\partial_{t}u\left(x_{1},x_{2},x_{3},vx_{1}\right)\right|^{2}dx\nonumber \\
\lesssim\int\left|\nabla_{x}u\left(x_{1},x_{2},x_{3},0\right)\right|^{2}+\left|\partial_{t}u\left(x_{1},x_{2},x_{3},0\right)\right|^{2}dx
\end{eqnarray}
and the inequality from the other side can be established after a
simple change of variable.

Denoting $\left[|x|\right]:=\sqrt{\left(1+\mu\right)^{2}x_{1}^{2}+x_{2}^{2}+x_{3}^{2}}$,
we define
\begin{equation}
L_{+}^{U}\left(u,\mu\right)=\int_{x_{1}>0}\left[|x|\right]\left|V\left(x,\mu x_{1}\right)\right|\left|u\left(x_{1},x_{2},x_{3},\mu x_{1}\right)\right|\left|u_{t}\left(x_{1},x_{2},x_{3},\mu x_{1}\right)\right|dx
\end{equation}
and 
\begin{equation}
E_{+}^{U}\left(u,\mu\right)=\int_{x_{1}>0}\left|\nabla_{x}u\left(x_{1},x_{2},x_{3},\mu x_{1}\right)\right|^{2}+\left|\partial_{t}u\left(x_{1},x_{2},x_{3},\mu x_{1}\right)\right|^{2}dx.
\end{equation}
Since $\left|v\right|<1$, from \eqref{eq:div1}, one has
\begin{eqnarray}
E_{+}^{U}\left(u,v\right)\lesssim\frac{1}{\sqrt{v^{2}+1}}\int_{Y_{T}^{+}}\left[-v\partial_{x_{1}}uu_{t}+\left(\frac{\left|u_{t}\right|^{2}}{2}+\frac{\left|u_{x}\right|^{2}}{2}\right)\right]\,dx\nonumber \\
\lesssim\int_{Y_{B}^{+}}\left(\frac{\left|u_{t}\right|^{2}}{2}+\frac{\left|u_{x}\right|^{2}}{2}\right)\,dx+\int_{Y^{+}}\left|V(x,t)uu_{t}\right|\,dxdt\nonumber \\
=E_{+}^{U}\left(u,0\right)+\int_{Y^{+}}\left|V(x,t)uu_{t}\right|\,dxdt.
\end{eqnarray}
To deal with
\[
\int_{Y^{+}}\left|V(x,t)uu_{t}\right|\,dxdt
\]
we apply a change of variable and Fubini's theorem,
\[
\int_{Y^{+}}\left|V(x,t)uu_{t}\right|\,dxdt\lesssim\int_{0}^{v}\left(L_{+}^{U}\left(u,\mu\right)\right)\,d\mu.
\]
Note that by the decay assumption on the potential and apply Hardy's
inequality,
\begin{align*}
L_{+}^{U}\left(u,\mu\right) & \lesssim\int\frac{1}{\left\langle \sqrt{\left(1+\mu\right)^{2}x_{1}^{2}+x_{2}^{2}+x_{3}^{2}}\right\rangle }\left|uu_{t}\right|dx\\
 & \lesssim E_{+}^{U}\left(u,\mu\right).
\end{align*}
Therefore using notations above, we obtain
\begin{equation}
E_{+}^{U}\left(u,v\right)\lesssim\frac{\sqrt{v^{2}+1}}{\left|1-\left|v\right|\right|}\left(E_{+}^{U}\left(u,0\right)+\int_{0}^{v}E_{+}^{U}\left(u,\mu\right)\,d\mu\right).
\end{equation}
By Gr\"onwall's inequality with respect to $v$, it follows that
\begin{equation}
E_{+}^{U}\left(u,v\right)\lesssim E_{+}^{U}\left(u,0\right)
\end{equation}
as desired provided $\left|v\right|<1$.
\end{proof}

\subsubsection{Inhomogeneous comparison}

In nonlinear applications, we also need to handle inhomogeneous equations.
So here we briefly discuss the energy comparison.

Suppose
\begin{equation}
\partial_{tt}u-\Delta u+V(x,t)u=F,
\end{equation}
then it is clear that
\begin{eqnarray}
Fu_{t} & = & u_{t}\left(\square u-V(t)u\right)\nonumber \\
 & = & -\partial_{t}\left(\frac{\left|u_{t}\right|^{2}}{2}+\frac{\left|u_{x}\right|^{2}}{2}\right)+\mathrm{div}\left(\nabla uu_{t}\right)-V(x,t)uu_{t}.\label{eq:vector-1}
\end{eqnarray}
We again apply the space-time divergence theorem to
\begin{equation}
\left(\nabla uu_{t},-\left(\frac{\left|u_{t}\right|^{2}}{2}+\frac{\left|u_{x}\right|^{2}}{2}\right)\right).
\end{equation}
\begin{thm}
\label{thm:generalC-1}Let $\left|v\right|<1$. Suppose
\begin{equation}
\partial_{tt}u-\Delta u+V(x,t)u=F(x,t)
\end{equation}
and
\begin{equation}
\left|V(x,\mu x_{1})\right|\lesssim\frac{1}{\left\langle x\right\rangle ^{2}}
\end{equation}
 for $0\leq\left|\mu\right|<1$. Then

\begin{eqnarray}
\int\left|\nabla_{x}u\left(x_{1},x_{2},x_{3},vx_{1}\right)\right|^{2}+\left|\partial_{t}u\left(x_{1},x_{2},x_{3},vx_{1}\right)\right|^{2}dx\nonumber \\
\text{\ensuremath{\lesssim}}\int\left|\nabla_{x}u\left(x_{1},x_{2},x_{3},0\right)\right|^{2}+\left|\partial_{t}u\left(x_{1},x_{2},x_{3},0\right)\right|^{2}dx\label{eq:generalC-1}\\
+\int_{\left|t\right|\leq\left|x\right|}\sqrt{\left|x\right|^{2}+t^{2}}\left|F(x,t)\right|^{2}dxdt\nonumber 
\end{eqnarray}
and
\begin{eqnarray}
\int\left|\nabla_{x}u\left(x_{1},x_{2},x_{3},0\right)\right|^{2}+\left|\partial_{t}u\left(x_{1},x_{2},x_{3},0\right)\right|^{2}dx\nonumber \\
\text{\ensuremath{\lesssim}}\int\left|\nabla_{x}u\left(x_{1},x_{2},x_{3},vx_{1}\right)\right|^{2}+\left|\partial_{t}u\left(x_{1},x_{2},x_{3},vx_{1}\right)\right|^{2}dx\label{eq:generalC-1}\\
+\int_{\left|t\right|\leq\left|x\right|}\sqrt{\left|x\right|^{2}+t^{2}}\left|F(x,t)\right|^{2}dxdt\nonumber 
\end{eqnarray}
where the implicit constant depends on $v$ and $V$.
\end{thm}

\begin{proof}
The proof follows the identical arguments as Theorem \ref{thm:generalC}.
The only change is that \eqref{eq:div1} is replaced by
\begin{eqnarray}
\frac{1}{\sqrt{v^{2}+1}}\int_{Y_{T}^{+}}\left[v\partial_{x_{1}}uu_{t}-\left(\frac{\left|u_{t}\right|^{2}}{2}+\frac{\left|u_{x}\right|^{2}}{2}\right)\right]\,dx+\int_{Y_{B}^{+}}\left(\frac{\left|u_{t}\right|^{2}}{2}+\frac{\left|u_{x}\right|^{2}}{2}\right)\,dx\nonumber \\
=\int_{Y^{+}}V(x,t)uu_{t}\,dxdt+\int_{Y^{+}}Fu_{t}\,dxdt,
\end{eqnarray}
where $n_{L}(x)$ is a vector of norm $1$.

All other terms can be handled as Lemma Theorem \ref{thm:generalC}
and the last term is bounded by
\begin{align*}
\int_{Y^{+}}\left|Fu_{t}\right|\,dxdt & \lesssim\int_{0}^{v}\int\,\left(\left[|x|\right]F\right)^{2}\,dxd\mu+\int_{0}^{v}\left(L_{+}^{U}\left(u,\mu\right)\right)\,d\mu\\
 & \lesssim\int_{\left|t\right|\leq\left|x\right|}\sqrt{\left|x\right|^{2}+t^{2}}\left|F(x,t)\right|^{2}dxdt+\int_{0}^{v}\left(L_{+}^{U}\left(u,\mu\right)\right)\,d\mu.
\end{align*}
Then one can apply Gr\"onwall's inequality with respect to $v$ by
the same way as Theorem \ref{thm:generalC} and conclude the desired
result.
\end{proof}
From the theorem above, we know initial energy with respect to different
frames stays comparable up to $\int_{\left|t\right|\leq\left|x\right|}\sqrt{\left|x\right|^{2}+t^{2}}\left|F(x,t)\right|^{2}dxdt$.

\subsection{Agmon's estimates via wave equations}

As a by product of Theorem \ref{thm:generalC}, we show Agmon's estimates
\cite{Agm} for the decay of eigenfunctions associated with negative
eigenvalues of
\begin{equation}
H=-\Delta+V.
\end{equation}
 Again, we restrict our attention to the class of potentials satisfying
the assumption 
\begin{equation}
\left|V(x)\right|\leq C_{V}\left(1+x^{2}\right)^{-1},\,\,\forall x\in\mathbb{R}^{3}.\label{eq:decay}
\end{equation}
As in Remark \ref{rem:dim}, all arguments and discussions are valid
for $x\in\mathbb{R}^{n}$.
\begin{thm}[Agmon]
\label{thm:Agmon}Let $V$ satisfy the assumption \eqref{eq:decay}.
Suppose $\phi\in H^{2}$
\begin{equation}
-\Delta\phi+V\phi=E\phi,\,\,E<0.\label{eq:eigen}
\end{equation}
Then $\forall\alpha\in[0,2\sqrt{-E})$
\begin{equation}
\int_{\mathbb{R}^{3}}e^{\alpha\left|x\right|}\left|\phi(x)\right|^{2}\,dx\simeq\int_{\mathbb{R}^{3}}\left|\phi(x)\right|^{2}\,dx,\label{eq:L2decay}
\end{equation}
with implicit constants depending on $\alpha,\,V$.

Furthermore, if $V\in H^{k}\mbox{\ensuremath{\left(\mathbb{R}^{3}\right)} }$where
$k>\frac{3}{2}$ and for $0\leq i\leq k$,
\begin{equation}
\left|\nabla^{i}V(x)\right|\leq C_{V,i}\left(1+x^{2}\right)^{-1}
\end{equation}
 then
\begin{equation}
\left|\phi(x)\right|\lesssim e^{-\frac{\alpha}{2}\left|x\right|}.\label{eq:pointwise}
\end{equation}
\end{thm}

\begin{proof}
It suffices to show $\forall\alpha\in[0,2\sqrt{-E})$
\[
\int_{\mathbb{R}^{3}}e^{\alpha\left|x_{j}\right|}\left|\phi(x)\right|^{2}\,dx\simeq\int_{\mathbb{R}^{3}}\left|\phi(x)\right|^{2}\,dx,\,\,\forall j=1,2,3.
\]
Without loss of generality, we pick $j=1$.

With Theorem \ref{thm:generalC}, we know if $u_{tt}+Hu=0$, then
with $\left|v\right|<1$
\begin{eqnarray}
\int\left|\nabla_{x}u\left(x_{1},x_{2},x_{3},vx_{1}\right)\right|^{2}+\left|\partial_{t}u\left(\gamma x_{1},x_{2},x_{3},vx_{1}\right)\right|^{2}dx\nonumber \\
\simeq\int\left|\nabla_{x}u\left(x_{1},x_{2},x_{3},0\right)\right|^{2}+\left|\partial_{t}u\left(x_{1},x_{2},x_{3},0\right)\right|^{2}dx.
\end{eqnarray}
We can rewrite the result above using half-wave operator $e^{it\sqrt{H}}$,
for $f\in L^{2}$ then
\begin{equation}
\int\left(\left|e^{ivt\sqrt{H}}f\right|_{t=x_{1}}^{2}\right)\,dx\simeq\int\left|f\right|^{2}\,dx.
\end{equation}
We pick $f=\phi$ satisfying
\begin{equation}
-\Delta\phi+V\phi=E\phi,\,\,E<0,
\end{equation}
then
\begin{equation}
\int e^{-vx_{1}2\sqrt{-E}}\left|\phi\right|^{2}\,dx=\int\left|e^{-vx_{1}\sqrt{-E}}\phi\right|^{2}\,dx\simeq\int\left|\phi\right|^{2}\,dx.
\end{equation}
With $v$ replaced by $-v$, we obtain
\begin{equation}
\int e^{vx_{1}2\sqrt{-E}}\left|\phi\right|^{2}\,dx=\int\left|e^{-vx_{1}\sqrt{-E}}\phi\right|^{2}\,dx\simeq\int\left|\phi\right|^{2}\,dx.
\end{equation}
Therefore,
\begin{equation}
\int e^{\left|2v\sqrt{-E}\right|\left|x_{1}\right|}\left|\phi\right|^{2}\,dx\simeq\int\left|\phi\right|^{2}\,dx.
\end{equation}
Fixed an $\alpha\in[0,2\sqrt{-E})$, we can find $\left|v\right|\in[0,1)$
such that $\alpha=\left|2v\sqrt{-E}\right|$, then it follows that
\begin{equation}
\int e^{\alpha\left|x_{1}\right|}\left|\phi\right|^{2}\,dx\simeq\int\left|\phi\right|^{2}\,dx.
\end{equation}
Therefore the estimate \eqref{eq:L2decay} is proved. 

Next we move to \eqref{eq:pointwise}. Since
\begin{equation}
-\Delta\phi+V\phi=E\phi,\,\,E<0,
\end{equation}
then
\begin{equation}
\int\left|\nabla\phi\right|^{2}dx+\int V\left|\phi\right|^{2}dx=E\int\left|\phi\right|^{2}dx,
\end{equation}
\begin{equation}
\int\left|\nabla\phi\right|^{2}dx\leq\left\Vert V\right\Vert _{L^{\infty}}\int\left|\phi\right|^{2}dx.
\end{equation}
Differentiating the equation, for any multi-index $\beta$
\begin{equation}
-\Delta\left(\partial^{\beta}\phi\right)+\partial^{\beta}\left(V\phi\right)=E\partial^{\beta}\phi
\end{equation}
we can conclude
\begin{equation}
\int\left|\nabla\left(\partial^{\beta}\phi\right)\right|^{2}dx\leq\int\partial^{\beta}\left(V\phi\right)\partial^{\beta}\phi\,dx.
\end{equation}
By induction, we obtain
\begin{equation}
\int\left|\nabla\left(\partial^{\beta}\phi\right)\right|^{2}dx\leq\left\Vert V\right\Vert _{W^{\left|\beta\right|,\infty}}\int\left|\phi\right|^{2}dx.
\end{equation}
Let $\psi$ be a smooth bump-cutoff function such that $\psi=1$ in
$B_{1}(0)$ and $\psi=0$ in $\mathbb{R}^{3}\backslash B_{2}(0)$.
We localize our estimate,
\begin{equation}
\int\left(-\Delta\phi(x)+V\phi(x)\right)\bar{\phi}(x)\psi^{2}(x-y)\,dx=E\int\left|\phi(x)\right|^{2}\psi^{2}(x-y)\,dx.
\end{equation}
Integrating by parts, we know
\begin{eqnarray}
\int\left(-\Delta\phi(x)+V\phi(x)\right)\bar{\phi}(x)\psi^{2}(x-y)\,dx\nonumber \\
=\int V\left|\phi(x)\right|^{2}\psi^{2}(x-y)\,dx\\
+\int\left|\nabla\phi(x)\right|^{2}\psi^{2}(x-y)\,dx\nonumber \\
+2\int\nabla\phi(x)\bar{\phi}(x)\psi(x-y)\nabla\psi(x-y)\,dx.\nonumber 
\end{eqnarray}
Therefore, by the Cauchy-Schwarz inequality,
\begin{eqnarray}
\int\left|\nabla\phi(x)\right|^{2}\psi^{2}(x-y)\,dx & \lesssim & E\int\left|\phi(x)\right|^{2}\psi^{2}(x-y)\,dx\nonumber \\
 &  & +\int V\left|\phi(x)\right|^{2}\psi^{2}(x-y)\,dx\\
 &  & +2\int\left|\phi(x)\nabla\psi(x-y)\right|^{2}\,dx.\nonumber 
\end{eqnarray}
It follows that
\begin{equation}
\sup_{y\in\mathbb{R}^{3}}\int_{\left|x-y\right|\leq1}\left|\nabla\phi(x)\right|^{2}\,dx\lesssim\left(\left\Vert V\right\Vert _{L^{\infty}}+1+\left|E\right|\right)\int_{\left|x-y\right|\leq2}\left|\phi(x)\right|^{2}\,dx.
\end{equation}
Inductively as above, we have
\begin{equation}
\sup_{y\in\mathbb{R}^{3}}\int_{\left|x-y\right|\leq1}\left|\nabla\left(\partial^{\beta}\phi\right)\right|^{2}\,dx\lesssim\left(\left\Vert V\right\Vert _{W^{\left|\beta\right|,\infty}}+1+\left|E\right|\right)\int_{\left|x-y\right|\leq2}\left|\phi(x)\right|^{2}\,dx.
\end{equation}
Finally by Sobolev's embedding theorem,
\begin{eqnarray}
\sup_{y\in\mathbb{R}^{3}}\sup_{\left|x-y\right|\leq1}\left|\phi(x)\right|^{2} & \lesssim & \sum_{\beta\leq k}\sup_{y\in\mathbb{R}^{3}}\int_{\left|x-y\right|\leq1}\left|\left(\partial^{\beta}\phi\right)\right|^{2}\,dx\nonumber \\
 & \lesssim & \int_{\left|x-y\right|\leq2}\left|\phi(x)\right|^{2}\,dx\\
 & \lesssim & e^{-\alpha\left(\left|y\right|-2\right)}\int_{\left|x-y\right|\leq2}e^{\alpha\left|x\right|}\left|\phi(x)\right|^{2}\,dx\nonumber \\
 & \lesssim & e^{-\alpha\left|y\right|}\int\left|\phi(x)\right|^{2}\,dx\nonumber \\
 & \lesssim & e^{-\alpha\left|y\right|}.\nonumber 
\end{eqnarray}
Hence,
\begin{equation}
\left|\phi(y)\right|\lesssim e^{-\frac{\alpha}{2}\left|y\right|}
\end{equation}
as claimed.
\end{proof}

\section{Endpoint Reversed Strichartz Estimates\label{sec:revered}}

In \cite{GC2}, we analyzed the endpoint reversed Strichartz estimates
along slanted lines for both homogeneous and inhomogeneous cases.
In this section, we will study the reversed Strichartz estimates along
general trajectories in several different settings.

Recall that a trajectory $\vec{Y}(t)\in\mathbb{R}^{3}$ is called
an admissible trajectory if $\vec{Y}(t)$ is $C^{1}$ and there exists
$0\leq\ell<1$ such $\left|\vec{Y}(t)\right|\leq\ell<1$ for $t\in\mathbb{R}$.

\subsection{Free wave equations}

In this subsection, we set
\begin{equation}
u(x,t)=\frac{\sin\left(t\sqrt{-\Delta}\right)}{\sqrt{-\Delta}}f+\cos\left(t\sqrt{-\Delta}\right)g+\int_{0}^{t}\frac{\sin\left(\left(t-s\right)\sqrt{-\Delta}\right)}{\sqrt{-\Delta}}F(s)\,ds
\end{equation}
and
\begin{equation}
u^{S}(x,t):=u\left(x+\vec{Y}(t),t\right).
\end{equation}
We first establish Theorem \ref{thm:reversedF}, the $L_{x}^{\infty}L_{t}^{2}$
endpoint of reversed estimates.
\begin{proof}[Proof of Theorem \ref{thm:reversedF}]
For the first term,
\begin{equation}
u_{1}(x,t)=\frac{\sin\left(t\sqrt{-\Delta}\right)}{\sqrt{-\Delta}}f=\frac{1}{4\pi t}\int_{\left|x-y\right|=t}f(y)\,\sigma\left(dy\right).
\end{equation}
So in polar coordinates,
\begin{eqnarray}
\left\Vert u_{1}^{S}(x,t)\right\Vert _{L_{t}^{2}[0,\infty)}^{2} & \lesssim & \int_{0}^{\infty}\left(\int_{\mathbb{S}}f(x+\vec{Y}(r)+r\omega)r\,d\omega\right)^{2}dr.
\end{eqnarray}
Up to translation, it suffices to estimate when $x=0$, so we consider
\[
\int_{0}^{\infty}\left(\int_{\mathbb{S}}f(\vec{Y}(r)+r\omega)r\,d\omega\right)^{2}dr.
\]
By Cauchy-Schwarz, one has

\begin{align}
\int_{0}^{\infty}\left(\int_{\mathbb{S}}f(\vec{Y}(r)+r\omega)r\,d\omega\right)^{2}dr & \lesssim\left(\int_{0}^{\infty}\int_{\mathbb{S}}f(\vec{Y}(r)+r\omega))^{2}r^{2}\,d\omega dr\right)\left(\int_{\mathcal{S}^{2}}d\omega\right).
\end{align}
Performing the change of variable that
\begin{equation}
\left(r,\omega\right)\rightarrow x'=\vec{Y}(r)+r\omega
\end{equation}
we compare the Jacobian of this change variable with the Jacobian
of the regular polar coordinate:
\begin{equation}
\left(r,\omega\right)\rightarrow x=r\omega.
\end{equation}
It is equivalent to show the change of variable
\begin{equation}
x\rightarrow x'=\vec{Y}\left(\left|x\right|\right)+x
\end{equation}
has a Jacobian which is bounded from above and below. 

Letting $\vec{Y}\left(\left|x\right|\right)=\left(Y_{1}\left(\left|x\right|\right),Y_{2}\left(\left|x\right|\right),Y_{3}\left(\left|x\right|\right)\right)$,
we compute the Jocobian and obtain
\begin{equation}
\frac{\partial x'}{\partial x}=I+\left\{ \begin{array}{ccc}
Y_{1}^{'}\left(\left|x\right|\right)\frac{x_{1}}{\left|x\right|} & Y_{1}^{'}\left(\left|x\right|\right)\frac{x_{2}}{\left|x\right|} & Y_{1}^{'}\left(\left|x\right|\right)\frac{x_{3}}{\left|x\right|}\\
Y_{2}^{'}\left(\left|x\right|\right)\frac{x_{1}}{\left|x\right|} & Y_{2}^{'}\left(\left|x\right|\right)\frac{x_{2}}{\left|x\right|} & Y_{2}^{'}\left(\left|x\right|\right)\frac{x_{3}}{\left|x\right|}\\
Y_{3}^{'}\left(\left|x\right|\right)\frac{x_{1}}{\left|x\right|} & Y_{3}^{'}\left(\left|x\right|\right)\frac{x_{2}}{\left|x\right|} & Y_{3}^{'}\left(\left|x\right|\right)\frac{x_{3}}{\left|x\right|}
\end{array}\right\} =:I+R.
\end{equation}
Then it is reduced to show that $R$ has an operator norm less than
$1$ uniformly with respect to $x$. 

Setting $\vec{d}(x)=\left(\frac{x_{1}}{\left|x\right|},\frac{x_{3}}{\left|x\right|},\frac{x_{3}}{\left|x\right|}\right)$,
we notice that $\forall\xi\in\mathbb{R}^{3},$
\begin{align}
\left|\xi R\xi^{T}\right| & =\xi\vec{Y}'\left(\left|x\right|\right)\otimes\vec{d}(x)\xi^{T}\nonumber \\
 & =\left\langle \vec{Y}'\left(\left|x\right|\right),\xi\right\rangle \left\langle \vec{d}(x),\xi\right\rangle \nonumber \\
 & \lesssim\left|\vec{Y}'\left(\left|x\right|\right)\right|\left|\vec{d}(x)\right|\left|\xi\right|^{2}\\
 & \lesssim\left|\vec{Y}'\left(\left|x\right|\right)\right|\left|\xi\right|^{2}.\nonumber 
\end{align}
Since $\left|\vec{Y}'\left(\left|x\right|\right)\right|\leq\ell<1$,
$R(x)$ has an operator norm not larger than $\ell$. Hence the Jacobian
$\frac{\partial x'}{\partial x}$ is bounded from above and below
uniformly. Therefore the Jacobian of the change of variable
\begin{equation}
\left(r,\omega\right)\rightarrow x'=\vec{Y}(r)+r\omega
\end{equation}
 is comparable with the Jacobian of
\begin{equation}
\left(r,\omega\right)\rightarrow x=r\omega
\end{equation}
which is $r^{3}$ uniformly. 

So we can conclude that
\begin{align}
\left(\int_{0}^{\infty}\int_{\mathbb{S}}f(\vec{Y}(r)+r\omega))^{2}r^{2}\,d\omega dr\right) & \lesssim\int\left|f\left(x'\right)\right|^{2}\,dx'\nonumber \\
 & \lesssim\left\Vert f\right\Vert _{L^{2}}^{2}.
\end{align}
A similar argument holds for
\begin{equation}
u_{2}(x,t)=\cos\left(t\sqrt{-\Delta}\right)g.
\end{equation}
Therefore
\begin{equation}
\left\Vert u_{1}^{S}\right\Vert _{L_{x}^{\infty}L_{t}^{2}}+\left\Vert u_{2}^{S}\right\Vert _{L_{x}^{\infty}L_{t}^{2}}\lesssim\|f\|_{L^{2}}+\|g\|_{\dot{H}^{1}}.
\end{equation}
In particular,
\begin{equation}
\left\Vert u_{1}\right\Vert _{L_{x}^{\infty}L_{t}^{2}}+\left\Vert u_{2}\right\Vert _{L_{x}^{\infty}L_{t}^{2}}\lesssim\|f\|_{L^{2}}+\|g\|_{\dot{H}^{1}}.
\end{equation}
as claimed.

Next, we consider the inhomogenous case,
\begin{equation}
D(x,t)=\int_{0}^{t}\frac{\sin\left((t-s)\sqrt{-\Delta}\right)}{\sqrt{-\Delta}}F(s)\,ds.
\end{equation}
For the standard case, we consider
\begin{eqnarray*}
\left\Vert \int_{0}^{t}\frac{\sin\left((t-s)\sqrt{-\Delta}\right)}{\sqrt{-\Delta}}F(s)\,ds\right\Vert _{L_{t}^{2}} & = & \left\Vert \int_{0}^{t}\int_{\left|x-y\right|=t-s}\frac{1}{\left|x-y\right|}F(y,s)\,\sigma\left(dy\right)ds\right\Vert _{L_{t}^{2}}\\
 & = & \left\Vert \int_{\left|x-y\right|\leq t}\frac{1}{\left|x-y\right|}F\left(y,t-\left|x-y\right|\right)\,dy\right\Vert _{L_{t}^{2}}\\
 & \lesssim & \int\frac{1}{\left|x-y\right|}\left\Vert F\left(y,t-\left|x-y\right|\right)\right\Vert _{L_{t}^{2}}dy\\
 & \lesssim & \sup_{x\in\mathbb{R}^{3}}\int\frac{1}{\left|x-y\right|}\left\Vert F\left(y,t\right)\right\Vert _{L_{t}^{2}}dy\\
 & \lesssim & \left\Vert F\right\Vert _{L_{x}^{\frac{3}{2},1}L_{t}^{2}}.
\end{eqnarray*}
Therefore, indeed,
\begin{equation}
\left\Vert D\right\Vert _{L_{x}^{\infty}L_{t}^{2}}\lesssim\left\Vert F\right\Vert _{L_{x}^{\frac{3}{2},1}L_{t}^{2}}.
\end{equation}
Actually, we have
\[
\left\Vert D\right\Vert _{L_{x}^{\infty}L_{t}^{p}}\lesssim\left\Vert F\right\Vert _{L_{x}^{\frac{3}{2},1}L_{t}^{p}},\,\,1\leq p\leq\infty.
\]
Now we consider the estimate along an admissible trajectory $\vec{Y}(t)\in\mathbb{R}^{3}$.

We first notice that from the discussion above or the argument in
Appendix D,
\begin{equation}
T:=\frac{e^{it\sqrt{-\Delta}}}{\sqrt{-\Delta}}
\end{equation}
is a bounded operator from $L_{x}^{2}$ to $L_{x}^{\infty}L_{t}^{2}$.
Also the operator $T^{S}$:
\begin{equation}
T^{S}f:=\left(Tf\right)^{S}=\left(\frac{e^{it\sqrt{-\Delta}}}{\sqrt{-\Delta}}f\right)^{S}
\end{equation}
is a bounded operator from $L_{x}^{2}$ to $L_{x}^{\infty}L_{t}^{2}$.

Writing down the inhomogeneous evolution explicitly, one has
\begin{align}
\left\Vert \int_{0}^{t}\frac{\sin\left((t-s)\sqrt{-\Delta}\right)}{\sqrt{-\Delta}}F(s)\,ds\right\Vert _{L_{t}^{2}} & =\left\Vert \int_{0}^{t}\int_{\left|x-y\right|=t-s}\frac{1}{\left|x-y\right|}F(y,s)\,\sigma\left(dy\right)ds\right\Vert _{L_{t}^{2}}\nonumber \\
 & \lesssim\left\Vert \int_{\left|x-y\right|\leq t}\frac{1}{\left|x-y\right|}F\left(y,t-\left|x-y\right|\right)\,dy\right\Vert _{L_{t}^{2}}
\end{align}
Therefore,
\begin{align}
\sup_{x\mathbb{\in R}^{3}}\left\Vert \int_{\left|x-y\right|\leq t}\frac{1}{\left|x-y\right|}F\left(y,t-\left|x-y\right|\right)\,dy\right\Vert _{L_{t}^{2}} & \lesssim\sup_{x\mathbb{\in R}^{3}}\left\Vert \int\frac{1}{\left|x-y\right|}\left|F\left(y,t-\left|x-y\right|\right)\,\right|dy\right\Vert _{L_{t}^{2}}\nonumber \\
 & \lesssim\sup_{x\mathbb{\in R}^{3}}\left\Vert \int_{0}^{\infty}\frac{\sin\left((t-s)\sqrt{-\Delta}\right)}{\sqrt{-\Delta}}\left|F(s)\right|\,ds\right\Vert _{L_{t}^{2}}\nonumber \\
 & \lesssim\sup_{x\mathbb{\in R}^{3}}\left\Vert \Re\left(TT^{*}\sqrt{-\Delta}\left|F\right|\right)\right\Vert _{L_{t}^{2}}.
\end{align}
Hence we know
\begin{align}
\sup_{x\mathbb{\in R}^{3}}\left\Vert D^{S}\left(x,t\right)\right\Vert _{L_{t}^{2}} & \lesssim\sup_{x\mathbb{\in R}^{3}}\left\Vert \Re\left(T^{S}T^{*}\sqrt{-\Delta}\left|F\right|\right)\right\Vert _{L_{t}^{2}}\nonumber \\
 & \lesssim\left\Vert \sqrt{-\Delta}\left|F\right|(x,t)\right\Vert _{L_{x}^{1}L_{t}^{2}}\nonumber \\
 & \lesssim\left\Vert \nabla F\right\Vert _{L_{x}^{1}L_{t}^{2}}.
\end{align}
If the trajectory does not change the direction, we can obtain an
estimate which does not require $\sqrt{-\Delta}F$ by a similar argument
to the estimates along slanted lines in \cite{GC2}. Without loss
of generality, we assume that the direction of the trajectory is along
$x_{1}$. Then
\begin{equation}
D^{S}(x,t)=\int_{0}^{t}\int_{\left|x+\vec{Y}(t)-y\right|=t-s}\frac{F(y,s)}{\left|x+\vec{Y}(t)-y\right|}\,\sigma\left(dy\right)ds
\end{equation}
and
\begin{eqnarray}
\left\Vert D^{S}(x,\cdot)\right\Vert _{L_{t}^{2}} & = & \left\Vert \int_{0}^{t}\int_{\left|x+\vec{Y}(t)-y\right|=t-s}\frac{F(y,s)}{\left|x+\vec{Y}(t)-y\right|}\,\sigma\left(dy\right)ds\right\Vert _{L_{t}^{2}}\nonumber \\
 & = & \left\Vert \int_{\left|y\right|\leq t}\frac{F(x+\vec{Y}(t)-y,t-\left|y\right|)}{\left|y\right|}\,dy\right\Vert _{L_{t}^{2}}\\
 & \leq & \left\Vert \int_{\mathbb{R}^{3}}\frac{\left|F(x-y,t-\left|y+\vec{Y}(t)\right|)\right|}{\left|y+\vec{Y}(t)\right|}\,dy\right\Vert _{L_{t}^{2}}\nonumber \\
 & \leq & \left\Vert \int_{\mathbb{R}^{3}}\frac{\left|F(x-y,t-\left|y+\vec{Y}(t)\right|)\right|}{\sqrt{y_{2}^{2}+y_{3}^{2}}}\,dy\right\Vert _{L_{t}^{2}},\nonumber 
\end{eqnarray}
where in the third line, we used a change of variable and for the
last inequality and reduce the norm of $y$ to the norm of the component
of $y$ orthogonal to the direction of the motion. 

Finally,
\begin{equation}
\left\Vert \int_{\mathbb{R}^{3}}\frac{F(x-y,t-\left|y+\vec{Y}(t)\right|)}{\sqrt{y_{2}^{2}+y_{3}^{2}}}\,dy\right\Vert _{L_{t}^{2}}\leq\int_{\mathbb{R}^{3}}\frac{\left\Vert F(x-y,t-\left|y+\vec{Y}(t)\right|)\right\Vert _{L_{t}^{2}}}{\sqrt{y_{2}^{2}+y_{3}^{2}}}\,dy
\end{equation}
For fixed $y$, if we apply a change of variable of $t$ here, the
Jacobian is bounded by $1-|\vec{Y}'|$ and $1+|\vec{Y}'|$, so
\begin{eqnarray}
\int_{\mathbb{R}^{3}}\frac{\left\Vert F(x-y,t-\left|y+\vec{Y}\left(t\right)\right|)\right\Vert _{L_{t}^{2}}}{\sqrt{y_{2}^{2}+y_{3}^{2}}}\,dy & \lesssim & \int_{\mathbb{R}^{3}}\frac{\left\Vert F(x-y,\cdot)\right\Vert _{L_{t}^{2}}}{\sqrt{y_{2}^{2}+y_{3}^{2}}}dy\nonumber \\
 & \lesssim & \left\Vert F\right\Vert _{L_{x_{1}}^{1}L_{\widehat{x_{1}}}^{2,1}L_{t}^{2}}
\end{eqnarray}
where $\widehat{x_{1}}$ denotes the subspace orthogonal to $x_{1}$
(more generally, the subspace orthogonal to the direction of the motion).
Here $L^{2,1}$ is the Lorentz norm and the last inequality follows
from H\"older's inequality of Lorentz spaces. Therefore,
\begin{equation}
\left\Vert D^{S}\right\Vert _{L_{x}^{\infty}L_{t}^{2}}\lesssim\left\Vert F\right\Vert _{L_{x_{1}}^{1}L_{\widehat{x_{1}}}^{2,1}L_{t}^{2}}.\label{eq:ersI-1}
\end{equation}
as claimed.

Finally, we consider the estimate with the source term $F$ along
an admissible trajectory. This follows from a duality or the same
argument as in \cite{GC2}. So we conclude that
\begin{equation}
\left\Vert D^{S}\right\Vert _{L_{x}^{\infty}L_{t}^{2}}\lesssim\left\Vert \nabla F^{S'}\right\Vert {}_{L_{x}^{1}L_{t}^{2}},
\end{equation}
and
\begin{equation}
\left\Vert D^{S}\right\Vert _{L_{x}^{\infty}L_{t}^{2}}\lesssim\left\Vert F^{S'}\right\Vert {}_{L_{x_{1}}^{1}L_{\widehat{x_{1}}}^{2,1}L_{t}^{2}}
\end{equation}
provided $\vec{Y}(t)$ moves along $x_{1}$.

The theorem is proved.
\end{proof}
\begin{rem}
We notice that by Sobolev's embedding, see \cite{CRT} and \cite{Tar}, one has
\[
\dot{W}_{x}^{1,1}\hookrightarrow L^{\frac{3}{2},1}.
\]
Therefore indeed, the estimates along general curves requires slightly
more regularity than  the standard cases.
\end{rem}

We have the other endpoint version of reversed space-time estimates
with the norm $L_{x}^{6,2}L_{t}^{\infty}$ in Theorem \ref{thm:reversedlocalF}.
\begin{proof}[Proof of  Theorem \ref{thm:reversedlocalF}]
Consider $t\ge0$ and define 
\begin{equation}
Tf=\frac{\sin\left(t\sqrt{-\Delta}\right)}{\sqrt{-\Delta}}f
\end{equation}
then 
\begin{equation}
T^{*}F=\int_{0}^{\infty}\frac{\sin\left(t\sqrt{-\Delta}\right)}{\sqrt{-\Delta}}F(t)\,dt,
\end{equation}
and
\begin{align}
TT^{*}F & =\int_{0}^{\infty}\frac{\sin\left(t\sqrt{-\Delta}\right)}{\sqrt{-\Delta}}\frac{\sin\left(s\sqrt{-\Delta}\right)}{\sqrt{-\Delta}}F(s)\,ds\\
 & =\frac{1}{2}\int_{0}^{\infty}\left(\frac{\cos\left(\left(t-s\right)\sqrt{-\Delta}\right)}{-\Delta}-\frac{\cos\left(\left(t+s\right)\sqrt{-\Delta}\right)}{-\Delta}\right)F(s)\,ds.\nonumber 
\end{align}
We compute the kernel of
\begin{equation}
\frac{\cos\left(h\sqrt{-\Delta}\right)}{-\Delta}F=\int_{\mathbb{R}^{3}}K(x,y,h)F(y)\,dy.
\end{equation}
By straightforward computations, one has
\begin{equation}
\frac{\cos\left(h\sqrt{-\Delta}\right)}{-\Delta}=\frac{1}{-\Delta}-\int_{0}^{h}\frac{\sin\left(s\sqrt{-\Delta}\right)}{\sqrt{-\Delta}}\,ds=\int_{h}^{\infty}\frac{\sin\left(s\sqrt{-\Delta}\right)}{\sqrt{-\Delta}}\,ds.
\end{equation}
By the explicit kernel of $\frac{\sin\left(s\sqrt{-\Delta}\right)}{\sqrt{-\Delta}}$,
we know that
\begin{equation}
K(x,y,h)=\begin{cases}
\frac{1}{\left|x-y\right|} & \left|x-y\right|\geq h\\
0 & \left|x-y\right|<h
\end{cases}.\label{eq:kernelC}
\end{equation}
Notice that in $\mathbb{R}^{3}$, $\frac{1}{\left|x\right|}\in L^{3,\infty}$,
so
\begin{equation}
\left\Vert \int_{0}^{\infty}\left(\frac{\cos\left(\left(t-s\right)\sqrt{-\Delta}\right)}{-\Delta}\right)F(s)\,ds\right\Vert _{L_{x}^{6,2}L_{t}^{\infty}}\lesssim\left\Vert F\right\Vert _{L_{x}^{\frac{6}{5},2}L_{t}^{1}}
\end{equation}
by Young's inequality for convolution. It follows that
\begin{equation}
\left\Vert Tf\right\Vert _{L_{x}^{6,2}L_{t}^{\infty}}=\left\Vert \frac{\sin\left(t\sqrt{-\Delta}\right)}{\sqrt{-\Delta}}f\right\Vert _{L_{x}^{6,2}L_{t}^{\infty}}\lesssim\left\Vert f\right\Vert _{L^{2}}.
\end{equation}
Now we consider the shifted version:
\begin{equation}
T^{S}f=\left(Tf\right)^{S}=\left(\frac{\sin\left(t\sqrt{-\Delta}\right)}{\sqrt{-\Delta}}f\right)^{S}
\end{equation}
From the computations above, the kernel of
\begin{equation}
T^{S}\left(T^{S}\right)^{*}
\end{equation}
can be written as two parts
\begin{equation}
K\left(x+\vec{Y}(t),y+\vec{Y}(s),t-s\right)+K\left(x+\vec{Y}(t),y+\vec{Y}(s),t+s\right).
\end{equation}
By \eqref{eq:kernelC}, we have
\begin{equation}
K\left(x+\vec{Y}(t),y+\vec{Y}(s),t-s\right)=\begin{cases}
\frac{1}{\left|x+\vec{Y}(t)-\left(y+\vec{Y}(s)\right)\right|} & \left|x+\vec{Y}(t)-\left(y+\vec{Y}(s)\right)\right|\geq\left|t-s\right|\\
0 & \left|x+\vec{Y}(t)-\left(y+\vec{Y}(s)\right)\right|<\left|t-s\right|
\end{cases}
\end{equation}
For $\left|x+\vec{Y}(t)-\left(y+\vec{Y}(s)\right)\right|\geq t-s$,
\begin{align}
\left|x+\vec{Y}(t)-\left(y+\vec{Y}(s)\right)\right| & \geq\left|x-y\right|-\left|\vec{Y}(t)-\vec{Y}(s)\right|\nonumber \\
 & \geq\left|x-y\right|-\ell\left|t-s\right|\\
 & \geq\left|x-y\right|-\ell\left|x+\vec{Y}(t)-\left(y+\vec{Y}(s)\right)\right|\nonumber 
\end{align}
Therefore,
\begin{equation}
\left|x-y\right|\lesssim\left|x-\vec{Y}(t)-\left(y-\vec{Y}(s)\right)\right|
\end{equation}
provided
\begin{equation}
\left|x+\vec{Y}(t)-\left(y+\vec{Y}(s)\right)\right|\geq t-s.
\end{equation}
Hence
\begin{equation}
\left|K\left(x+\vec{Y}(t),y+\vec{Y}(s),t-s\right)\right|_{L_{s}^{\infty}}\lesssim\frac{1}{\left|x-y\right|}.\label{eq:SfirstKe}
\end{equation}
For the second kernel, by similar computations, one has
\begin{equation}
K\left(x+\vec{Y}(t),y+\vec{Y}(s),t+s\right)=\begin{cases}
\frac{1}{\left|x+\vec{Y}(t)-\left(y+\vec{Y}(s)\right)\right|} & \left|x+\vec{Y}(t)-\left(y+\vec{Y}(s)\right)\right|\geq t+s\\
0 & \left|x+\vec{Y}(t)-\left(y+\vec{Y}(s)\right)\right|<t+s
\end{cases}.
\end{equation}
If $\left|x+\vec{Y}(t)-\left(y+\vec{Y}(s)\right)\right|\geq t+s$,
\begin{align}
\left|x+\vec{Y}(t)-\left(y+\vec{Y}(s)\right)\right| & \geq\left|x-y\right|-\left|\vec{Y}(t)-\vec{Y}(s)\right|\nonumber \\
 & \geq\left|x-y\right|-\ell\left|t-s\right|\\
 & \geq\left|x-y\right|-\ell\left|t+s\right|\nonumber \\
 & \left|\geq x-y\right|-\ell\left|x+\vec{Y}(t)-\left(y+\vec{Y}(s)\right)\right|.\nonumber 
\end{align}
Hence
\begin{equation}
\left|x-y\right|\lesssim\left|x+\vec{Y}(t)-\left(y+\vec{Y}(s)\right)\right|
\end{equation}
provided
\begin{equation}
\left|x+\vec{Y}(t)-\left(y+\vec{Y}(s)\right)\right|\geq t+s.
\end{equation}
Therefore,
\begin{equation}
\left|K\left(x+\vec{Y}(t),y+\vec{Y}(s),t+s\right)\right|_{L_{s}^{\infty}}\lesssim\frac{1}{\left|x-y\right|}.\label{eq:SsecondKe}
\end{equation}
By estimates \eqref{eq:SfirstKe} and \eqref{eq:SsecondKe}, we conclude
that
\begin{equation}
\left\Vert \left|T^{S}\left(T^{S}\right)^{*}F\right|\left(x,\cdot\right)\right\Vert _{L_{t}^{\infty}}\lesssim\int\frac{1}{\left|x-y\right|}\left\Vert F(y,\cdot)\right\Vert _{L_{t}^{1}}\,dy
\end{equation}
and
\begin{equation}
\left\Vert T^{S}\left(T^{S}\right)^{*}F\right\Vert _{L_{x}^{6,2}L_{t}^{\infty}}\lesssim\left\Vert F(y,\cdot)\right\Vert _{L_{x}^{\frac{6}{5},2}L_{t}^{1}}.\label{eq:local3}
\end{equation}
Therefore,
\begin{equation}
u_{1}(x,t):=Tf=\frac{\sin\left(t\sqrt{-\Delta}\right)}{\sqrt{-\Delta}}f
\end{equation}
satisfies
\begin{equation}
\left\Vert u_{1}^{S}(x,t)\right\Vert _{L_{x}^{6,2}L_{t}^{\infty}}=\left\Vert u_{1}(x+\vec{Y}(t),t)\right\Vert _{L_{x}^{6,2}L_{t}^{\infty}}\lesssim\left\Vert f\right\Vert _{L^{2}}.\label{eq:local1}
\end{equation}
By a similar argument, we have that
\begin{equation}
u_{2}(x,t)=\cos\left(t\sqrt{-\Delta}\right)g
\end{equation}
satisfies
\begin{equation}
\left\Vert u_{2}^{S}(x,t)\right\Vert _{L_{x}^{6,2}L_{t}^{\infty}}=\left\Vert u_{2}(x+\vec{v}(t),t)\right\Vert _{L_{x}^{6,2}L_{t}^{\infty}}\lesssim\left\Vert g\right\Vert _{\dot{H}^{1}}.\label{eq:local2}
\end{equation}
For the inhomogeneous case, we again consider
\begin{equation}
D(x,t)=\int_{0}^{t}\frac{\sin\left((t-s)\sqrt{-\Delta}\right)}{\sqrt{-\Delta}}F(s)\,ds.
\end{equation}
For the standard case, as above,
\begin{eqnarray}
\left\Vert \int_{0}^{t}\frac{\sin\left((t-s)\sqrt{-\Delta}\right)}{\sqrt{-\Delta}}F(s)\,ds\right\Vert _{L_{t}^{\infty}} & \lesssim & \int\frac{1}{\left|x-y\right|}\left\Vert F\left(y,\cdot\right)\right\Vert _{L_{t}^{2}}dy,
\end{eqnarray}
so
\begin{equation}
\left\Vert \int_{0}^{t}\frac{\sin\left((t-s)\sqrt{-\Delta}\right)}{\sqrt{-\Delta}}F(s)\,ds\right\Vert _{L_{x}^{6,2}L_{t}^{\infty}}\lesssim\left\Vert F\right\Vert _{L_{x}^{\frac{6}{5},2}L_{t}^{\infty}}.
\end{equation}
Two sum three pieces up, we conclude that
\begin{equation}
\left\Vert u\right\Vert _{L_{x}^{6,2}L_{t}^{\infty}}\lesssim\|f\|_{L^{2}}+\|g\|_{\dot{H}^{1}}+\left\Vert F\right\Vert _{L_{x}^{\frac{6}{5},2}L_{t}^{\infty}}.
\end{equation}
For \eqref{eq:localrevSS-2}, it follows from estimates \eqref{eq:local1},
\eqref{eq:local2} and \eqref{eq:local3} with the same argument as \eqref{eq:freeesti11}.
Therefore,
\[
\left\Vert u^{S}(x,t)\right\Vert _{L_{x}^{6,2}L_{t}^{\infty}}\lesssim\|f\|_{L^{2}}+\|g\|_{\dot{H}^{1}}+\left\Vert F\right\Vert _{\dot{W}_{x}^{1,\frac{6}{5}}L_{t}^{1}}
\]
as claimed.
\end{proof}
\begin{rem}
From the embedding of Lorentz spaces, from
\begin{equation}
\left\Vert u\right\Vert _{L_{x}^{6,2}L_{t}^{\infty}}\lesssim\|f\|_{L^{2}}+\|g\|_{\dot{H}^{1}}+\left\Vert F\right\Vert _{L_{x}^{\frac{6}{5},2}L_{t}^{\infty}}
\end{equation}
one has
\begin{equation}
\left\Vert u\right\Vert _{L_{x}^{6}L_{t}^{\infty}}\lesssim\|f\|_{L^{2}}+\|g\|_{\dot{H}^{1}}+\left\Vert F\right\Vert _{L_{x}^{\frac{6}{5}}L_{t}^{\infty}},
\end{equation}
similarly,
\begin{equation}
\left\Vert u^{S}(x,t)\right\Vert _{L_{x}^{6}L_{t}^{\infty}}\lesssim\|f\|_{L^{2}}+\|g\|_{\dot{H}^{1}}+\left\Vert \nabla F\right\Vert _{L_{x}^{\frac{6}{5}}L_{t}^{1}}.
\end{equation}
\end{rem}

\subsection{Perturbed wave equations.}

Finally, we extend all of our estimates to the perturbed Hamiltonian.
In \cite{GC2}, we relied on Duhamel expansion of the perturbed evolution,
the estimates along trajectories for free ones and the standard estimates
for the perturbed ones. Here we present an alternative approach based
on the structure formula of the wave operators as in Section \ref{sec:Prelim}.
We only present the standard cases in Theorem \ref{thm:reversedP}
and other estimates can be obtained similarly. 

In this section, we suppose
\begin{equation}
H=-\Delta+V
\end{equation}
satisfies Definition \ref{def:potenialWO} and set
\begin{equation}
u(x,t)=\frac{\sin\left(t\sqrt{H}\right)}{\sqrt{H}}P_{c}f+\cos\left(t\sqrt{H}\right)P_{c}g+\int_{0}^{t}\frac{\sin\left(\left(t-s\right)\sqrt{H}\right)}{\sqrt{H}}P_{c}F(s)\,ds
\end{equation}
with
\begin{equation}
u^{S}(x,t):=u\left(x+\vec{Y}(t),t\right),
\end{equation}
where $P_{c}$ is the projection onto the continuous spectrum of $H$.
\begin{proof}[Proof of Theorem \ref{thm:reversedP}]
It suffices to consider
\begin{equation}
\frac{\sin\left(t\sqrt{H}\right)}{\sqrt{H}}P_{c}f.
\end{equation}
By construction,
\begin{equation}
\frac{\sin\left(t\sqrt{H}\right)}{\sqrt{H}}P_{c}f=W^{+}\frac{\sin\left(t\sqrt{-\Delta}\right)}{\sqrt{-\Delta}}\left(W^{+}\right)^{*}P_{c}f.
\end{equation}
Denoting
\begin{equation}
h=\left(W^{+}\right)^{*}P_{c}f,
\end{equation}
we have
\begin{equation}
\|P_{c}f\|_{L^{2}}\simeq\|h\|_{L^{2}}.
\end{equation}
Setting
\begin{equation}
G=\frac{\sin\left(t\sqrt{-\Delta}\right)}{\sqrt{-\Delta}}h,
\end{equation}
by Theorem \ref{thm:structure}, it is sufficient to consider the
boundedness of
\begin{equation}
G+\int_{\mathbb{S}^{2}}\int_{\mathbb{R}^{3}}g(x,y,\tau)G\left(S_{\tau}x+y\right)\,dyd\tau.
\end{equation}
Clearly, by the endpoint reversed Strichartz estimate for the free
case,
\begin{equation}
\|G\|_{L_{x}^{\infty}L_{t}^{2}}\lesssim\|h\|_{L^{2}}\simeq\|P_{c}f\|_{L^{2}}.
\end{equation}
Next, by Minkowski's inequality,
\begin{eqnarray}
\left\Vert \int_{\mathbb{S}^{2}}\int_{\mathbb{R}^{3}}g(x,y,\tau)G\left(S_{\tau}x+y\right)\,dyd\tau\right\Vert {}_{L_{x}^{\infty}L_{t}^{2}}\\
\lesssim\int_{\mathbb{S}^{2}}\int_{\mathbb{R}^{3}}\left\Vert g(x,y,\tau)G\left(S_{\tau}x+y\right)\right\Vert {}_{L_{x}^{\infty}L_{t}^{2}}dyd\tau\nonumber 
\end{eqnarray}
\begin{equation}
\left\Vert g(x,y,\tau)G\left(S_{\tau}x+y\right)\right\Vert {}_{L_{x}^{\infty}L_{t}^{2}}\lesssim\left\Vert g(x,y,\tau)\right\Vert _{L_{x}^{\infty}}\left\Vert G\left(S_{\tau}x+y\right)\right\Vert {}_{L_{x}^{\infty}L_{t}^{2}}.
\end{equation}
Since reflections with respect to a fixed plane and translations commute
with the solution of a free wave equation, we obtain
\begin{equation}
G\left(S_{\tau}x+y\right)=\frac{\sin\left(t\sqrt{-\Delta}\right)}{\sqrt{-\Delta}}h\left(S_{\tau}x+y\right).
\end{equation}
Therefore,
\begin{equation}
\left\Vert G\left(S_{\tau}x+y\right)\right\Vert {}_{L_{x}^{\infty}L_{t}^{2}}\lesssim\|h\left(S_{\tau}x+y\right)\|_{L^{2}}\lesssim\|h\|_{L^{2}}\simeq\|P_{c}f\|_{L^{2}}.
\end{equation}
It follows
\begin{eqnarray}
\left\Vert G+\int_{\mathbb{S}^{2}}\int_{\mathbb{R}^{3}}g(x,y,\tau)G\left(S_{\tau}x+y\right)\,dyd\tau\right\Vert {}_{L_{x}^{\infty}L_{t}^{2}}\\
\lesssim\left(1+\int_{\mathbb{S}^{2}}\int_{\mathbb{R}^{3}}\left\Vert g(x,y,\tau)\right\Vert _{L_{x}^{\infty}}dyd\tau\right)\|P_{c}f\|_{L^{2}}\lesssim\|f\|_{L^{2}}.\nonumber 
\end{eqnarray}
Then we conclude
\begin{equation}
\left\Vert \frac{\sin\left(t\sqrt{H}\right)}{\sqrt{H}}P_{c}f\right\Vert _{_{L_{x}^{\infty}L_{t}^{2}}}\lesssim\|f\|_{L^{2}},
\end{equation}
as claimed.
\end{proof}
For the other endpoint revered type estimate, Theorem \ref{thm:reveredLocalP}
follows in the same manner. 

\subsection{Wave equations with moving potentials}

Finally in this section, we consider the wave equation
\begin{equation}
\partial_{tt}u-\Delta u+V\left(x-\vec{\mu}t\right)u=0
\end{equation}
\begin{equation}
u(x,0)=g(x),\,u_{t}(x,0)=f(x)
\end{equation}
where the potential satisfies Definition \ref{def:potenialWO-1}.
Again without of loss of generality, we assume $\vec{\mu}$ is along
$\vec{e}_{1}$ and $\vec{\mu}<1$. Recall that associated to this
model, we define
\begin{equation}
H=-\Delta+V\left(\sqrt{1-\left|\vec{\mu}\right|^{2}}x_{1},x_{2},x_{3}\right).
\end{equation}

Let $m_{1},\,\ldots,\,m_{w}$ be the normalized bound states of $H$
associated to the negative eigenvalues $-\lambda_{1}^{2},\,\ldots,\,-\lambda_{w}^{2}$
respectively (notice that by our assumptions, $0$ is not an eigenvalue).
We denote by $P_{b}$ the projections on the the bound states of $H$
, respectively, and let $P_{c}=Id-P_{b}$. 

Performing a Lorentz transformation $L$ with respect to the moving
frame $\left(x-\vec{\mu}t,t\right)$, we have
\begin{equation}
\partial_{t't'}u_{L}+Hu_{L}=0,
\end{equation}
\begin{equation}
u_{L}(x',0)=\tilde{g}(x'),\,\left(u_{L}\right)_{t}(x',0)=\tilde{f}(x')
\end{equation}
and
\begin{equation}
\|f\|_{L^{2}}+\|g\|_{\dot{H}^{1}}\simeq\|\tilde{f}\|_{L^{2}}+\|\tilde{g}\|_{\dot{H}^{1}}.
\end{equation}
We can write
\begin{equation}
u_{L}\left(x',t'\right)=\sum_{i=1}^{w}a_{i}(t')m_{i}(x')+r_{L}\left(x',t'\right),
\end{equation}
such that
\begin{equation}
P_{c}r_{L}=r_{L}.
\end{equation}
Return to our original coordinate, we have a decomposition for $u$
that
\begin{equation}
u(x,t)=\sum_{i=1}^{w}a_{i}\left(\gamma(t-vx_{1})\right)\left(m_{i}\right)_{\mu}\left(x,t\right)+r\left(x,t\right)\label{eq:decomp}
\end{equation}
where
\begin{equation}
\left(m_{i}\right)_{\mu}(x,t)=m_{i}\left(\gamma\left(x_{1}-\mu t\right),x_{2},x_{3}\right).
\end{equation}
\begin{cor}
\label{cor:endmove}Let $\vec{Y}(t)\in\mathbb{R}^{3}$ be an admissible
trajectory. With the notations from above, we have 
\begin{equation}
\left\Vert r^{S}\right\Vert _{L_{x}^{\infty}L_{t}^{2}}\lesssim\|f\|_{L^{2}}+\|g\|_{\dot{H}^{1}},
\end{equation}
in particular,
\begin{equation}
\int_{0}^{\infty}\int_{\mathbb{R}^{3}}\frac{1}{\left\langle x-\vec{Y}(t)\right\rangle ^{\alpha}}r^{2}(x,t)\,dxdt\lesssim\|f\|_{L^{2}}+\|g\|_{\dot{H}^{1}}.
\end{equation}
\end{cor}

\begin{proof}
Notice that if $\vec{Y}(t)$ is an admissible trajectory in our original
frame $\left(x,t\right)$, then if we perform a Lorentz transformation
$L(\vec{\mu})$, in the new frame, the trajectory $\vec{Y}(t)$ can
be written as $\vec{Z}(t')$ with $\left|\vec{Z}'(t')\right|<\phi\left(\lambda,\vec{\ell}\right)<1.$
In other words, in the new coordinate, the trajectory is still admissible.
Then for fixed $x\in\mathbb{R}^{3}$,
\begin{equation}
\int\left|r^{S}(x,t)\right|^{2}dt\lesssim\sup_{x'\in\mathbb{R}^{3}}\int\left|r_{L}^{S'}(x',t')\right|^{2}dt',
\end{equation}
where
\begin{equation}
r_{L}^{S'}\left(x',t'\right)=r_{L}\left(x'+\vec{Z}(t'),t'\right).
\end{equation}
By construction and Theorem \ref{thm:generalC},
\begin{equation}
\sup_{x'\in\mathbb{R}^{3}}\int\left|r_{L}^{S'}(x',t')\right|^{2}dt'\lesssim\left(\|\tilde{f}\|_{L^{2}}+\|\tilde{g}\|_{\dot{H}^{1}}\right)^{2}\simeq\left(\|f\|_{L^{2}}+\|g\|_{\dot{H}^{1}}\right)^{2}
\end{equation}
and hence
\begin{equation}
\left\Vert r^{S}\right\Vert _{L_{x}^{\infty}L_{t}^{2}}\lesssim\|f\|_{L^{2}}+\|g\|_{\dot{H}^{1}}.
\end{equation}
The claim is proved.
\end{proof}

\section{Strichartz Estimates and Energy Estimates\label{sec:one}}

In this section, we establish Strichartz estimates and energy estimates
for scattering states to the wave equation
\begin{equation}
\partial_{tt}u-\Delta u+V\left(x-\vec{Y}(t)\right)u=0,
\end{equation}
\[
u(x,0)=g(x),\,u_{t}(x,0)=f(x)
\]
with
\begin{equation}
\left|\vec{Y}(t)-\vec{\mu}t\right|\lesssim\left\langle t\right\rangle ^{-\beta},\,\beta>1,\,\left|\vec{\mu}\right|<1.
\end{equation}
To simplify the problem, we assume
\begin{equation}
H=-\Delta+V\left(\sqrt{1-\left|\vec{\mu}\right|^{2}}x_{1},x_{2},x_{3}\right)
\end{equation}
only has one bound state $m$ such that
\begin{equation}
Hm=-\lambda^{2}m,\,\lambda>0.
\end{equation}
One can observe that our arguments work for the general case.

We start with reversed Strichartz estimates, Theorem \ref{thm:EndRStri}.
\begin{proof}[Proof of  Theorem \ref{thm:EndRStri}]
First of all, we need to understand the evolution of bound states.
Writing the equation as
\begin{equation}
\partial_{tt}u-\Delta u+V\left(x-\vec{\mu t}\right)u=\left[V\left(x-\vec{\mu}t\right)-V\left(x+\vec{Y}(t)\right)\right]u.\label{eq:Veq}
\end{equation}
Recall that we assume $\vec{\mu}$ is along $x_{1}$. Suppose $u(x,t)$
is a scattering state. As in \eqref{eq:decomp}, we decompose the evolution
as following,
\begin{equation}
u(x,t)=a\left(\gamma(t-\mu x_{1})\right)m_{\mu}\left(x,t\right)+r(x,t)\label{eq:evolution}
\end{equation}
where
\begin{equation}
m_{\mu}(x,t)=m\left(\gamma\left(x_{1}-\mu t\right),x_{2},x_{3}\right)
\end{equation}
and
\begin{equation}
P_{c}\left(H\right)r_{L}=r_{L}.
\end{equation}
Performing the Lorentz transformation $L$ with respect to the moving
frame $\left(x-\vec{\mu}t,t\right)$, we have
\begin{equation}
u_{L}(x',t')=a\left(t'\right)m\left(x'\right)+r_{L}(x',t'),\label{eq:evolutionL}
\end{equation}
and
\begin{equation}
\partial_{t't'}u_{L}+Hu_{L}=-M(x',t')u_{L}\label{eq:eqL}
\end{equation}
where
\begin{equation}
M(x',t')=-\left[V\left(x-\vec{\mu t}\right)-V\left(x+\vec{v}(t)\right)\right]_{L}.
\end{equation}
When $u$ is a scattering state in the sense Definition \ref{AO},
the scattering condition forces $a(t)$ to go $0$.

Plugging the evolution \eqref{eq:evolutionL} into the equation \eqref{eq:eqL}
and taking inner product with $m$, we get
\begin{equation}
\ddot{a}(t')-\lambda^{2}a(t')+a(t')\left\langle Mm,m\right\rangle +\left\langle Mr_{L},m\right\rangle =0
\end{equation}
Notice that
\begin{equation}
\left|M(x',t')\right|\lesssim\frac{1}{\left\langle \gamma\left(t'+\mu x_{1}'\right)\right\rangle ^{\beta}}.
\end{equation}
One can write
\begin{equation}
\ddot{a}(t')-\lambda^{2}a(t')+a(t')c(t')+h(t')=0,\label{eq:aode}
\end{equation}
\begin{equation}
c(t'):=\left\langle Mm,m\right\rangle 
\end{equation}
and
\begin{equation}
h(t'):=\left\langle Mr_{L},m\right\rangle .
\end{equation}
Since $w$ is exponentially localized by Agmon's estimate, we know
\begin{equation}
\left|c(t')\right|\lesssim e^{-b\left|t'\right|},\,b>0.
\end{equation}
The existence of the solution to the ODE \eqref{eq:aode} is clear.
We study the long-time behavior of the solution. Write the equation
as

\begin{equation}
\ddot{a}(t')-\lambda^{2}a(t')=-\left[a(t')c(t')+h(t')\right],
\end{equation}
and denote
\begin{equation}
N(t'):=-\left[a(t')c(t')+h(t')\right].
\end{equation}
Then
\begin{equation}
a(t')=\frac{e^{\lambda t'}}{2}\left[a(0)+\frac{1}{\lambda}\dot{a}(0)+\frac{1}{\lambda}\int_{0}^{t'}e^{-\lambda s}N(s)\,ds\right]+R(t')
\end{equation}
where
\begin{equation}
\left|R(t')\right|\lesssim e^{-ct'},
\end{equation}
for some positive constant $c>0$. Therefore, the stability condition
forces
\begin{equation}
a(0)+\frac{1}{\lambda}\dot{a}(0)+\frac{1}{\lambda}\int_{0}^{\infty}e^{-\lambda s}N(s)\,ds=0.\label{eq:stability}
\end{equation}
Then under the stability condition \eqref{eq:stability},
\begin{equation}
a(t')=e^{-\lambda t'}\left[a(0)+\frac{1}{2\lambda}\int_{0}^{\infty}e^{-\lambda s}N(s)ds\right]+\frac{1}{2\lambda}\int_{0}^{\infty}e^{-\lambda\left|t-s\right|}N(s)\,ds.
\end{equation}
By Young's inequality, to estimate all $L^{p}$ norms of $a(t')$,
it suffices to estimate the $L^{1}$ norm of $h(t')$, see \cite{GC2}. 

By Cauchy-Schwarz and Corollary \ref{cor:endmove},
\[
\int_{0}^{\infty}\left|\left\langle Mr_{L},m\right\rangle \right|dt\lesssim\left\Vert r_{L}\right\Vert _{L_{x'}^{\infty}L_{t'}^{2}}\lesssim\|f\|_{L^{2}}+\|g\|_{\dot{H}^{1}}.
\]
Therefore,
\begin{equation}
\left\Vert a(t)\right\Vert _{L^{p}[0,\infty)}\lesssim_{p}\|f\|_{L^{2}}+\|g\|_{\dot{H}^{1}}.
\end{equation}
Given $\vec{Z}(t)$ an admissible trajectory, set
\begin{equation}
B(x,t)=a\left(\gamma(t-\mu x_{1})\right)m_{\mu}\left(x,t\right),
\end{equation}
\begin{equation}
B^{S}(x,t)=B\left(x+\vec{Z}(t),t\right).
\end{equation}
By Agmon's estimate, see Theorem \ref{thm:Agmon}, and the $L^{1}$
norm estimate for $a(t')$, we have
\begin{equation}
\left\Vert B^{S}(x,t)\right\Vert _{L_{x}^{\infty}L_{t}^{2}[0,\infty)}\lesssim\|f\|_{L^{2}}+\|g\|_{\dot{H}^{1}}.
\end{equation}
By Corollary \ref{cor:endmove}, we also know
\begin{equation}
\left\Vert r^{S}(x,t)\right\Vert _{L_{x}^{\infty}L_{t}^{2}[0,\infty)}\lesssim\|f\|_{L^{2}}+\|g\|_{\dot{H}^{1}}
\end{equation}
Therefore, one has
\begin{equation}
\left\Vert u^{S}(x,t)\right\Vert _{L_{x}^{\infty}L_{t}^{2}[0,\infty)}\lesssim\|f\|_{L^{2}}+\|g\|_{\dot{H}^{1}}.
\end{equation}
We notice that this in particular implies for $\alpha>3$,
\begin{equation}
\int_{0}^{\infty}\int_{\mathbb{R}^{3}}\frac{1}{\left\langle x-\vec{Z}(t)\right\rangle ^{\alpha}}u^{2}(x,t)\,dxdt\lesssim\|f\|_{L^{2}}+\|g\|_{\dot{H}^{1}}.\label{eq:moveweighted}
\end{equation}
The theorem is proved.
\end{proof}
Next, we show Strichartz estimates, Theorem \ref{thm:Stri}, following
\cite{RS,LSch,GC2}. In the following, we use the short-hand notation
\begin{equation}
L_{t}^{p}L_{x}^{q}:=L_{t}^{p}\left([0,\infty),\,L_{x}^{q}\right).
\end{equation}
\begin{proof}[Proof of  Theorem \ref{thm:Stri}]
Following \cite{LSch}, we set $A=\sqrt{-\Delta}$ and notice that
\begin{equation}
\left\Vert Af\right\Vert _{L^{2}}\simeq\left\Vert f\right\Vert _{\dot{H}^{1}},\,\,\forall f\in C^{\infty}\left(\mathbb{R}^{3}\right).\label{eq:nabla}
\end{equation}
For real-valued $u=\left(u_{1},u_{2}\right)\in\mathcal{H}=\dot{H}^{1}\left(\mathbb{R}^{3}\right)\times L^{2}\left(\mathbb{\mathbb{R}}^{3}\right)$,
we write
\begin{equation}
U:=Au_{1}+iu_{2}.
\end{equation}
From \eqref{eq:nabla}, we know
\begin{equation}
\left\Vert U\right\Vert _{L^{2}}\simeq\left\Vert \left(u_{1},u_{2}\right)\right\Vert _{\mathcal{H}}.
\end{equation}
We also notice that $u$ solves the original wave equation if and
only if
\begin{equation}
U:=Au+i\partial_{t}u
\end{equation}
satisfies
\begin{equation}
i\partial_{t}U=AU+V\left(x-\vec{Y}(t)\right)u,
\end{equation}
\begin{equation}
U(0)=Ag+if\in L^{2}\left(\mathbb{R}^{3}\right).
\end{equation}
By Duhamel's formula,
\begin{equation}
U(t)=e^{itA}U(0)-i\int_{0}^{t}e^{-i\left(t-s\right)A}V\left(\cdot-\vec{Y}(s)\right)u(s)\,ds.
\end{equation}
Let $P:=A^{-1}\Re$, then from Strichartz estimates for the free evolution,
\begin{equation}
\left\Vert Pe^{itA}U(0)\right\Vert _{L_{t}^{p}L_{x}^{q}}\lesssim\left\Vert U(0)\right\Vert _{L^{2}}.\label{eq:Sfirst}
\end{equation}
Writing $V=V_{1}V_{2}$ and with the Christ-Kiselev lemma, Lemma \ref{lem:Christ-Kiselev},
it suffices to bound
\begin{equation}
\left\Vert P\int_{0}^{\infty}e^{-i\left(t-s\right)A}V_{1}V_{2}\left(\cdot-\vec{Y}(s)\right)u(s)\,ds\right\Vert _{L_{t}^{p}L_{x}^{q}}.
\end{equation}
We only need to analyze
\[
\left\Vert P\int_{0}^{\infty}e^{-i\left(t-s\right)A}V_{1}V_{2}\left(\cdot-\vec{Y}(s)\right)u(s)\,ds\right\Vert _{L_{t}^{p}L_{x}^{q}}\leq\left\Vert \widetilde{K}\right\Vert _{L_{t,x}^{2}\rightarrow L_{t}^{p}L_{x}^{q}}\left\Vert V_{2}\left(x-\vec{v}(s)\right)u\right\Vert _{L_{t,x}^{2}}
\]
where
\begin{equation}
\left(\widetilde{K}F\right)(t):=P\int_{0}^{\infty}e^{-i\left(t-s\right)A}V_{1}\left(\cdot-\vec{Y}(s)\right)F(s)\,ds.
\end{equation}
To show $\left\Vert \widetilde{K}\right\Vert _{L_{t,x}^{2}\rightarrow L_{t}^{p}L_{x}^{q}}$
is bounded, we test it against $F\in L_{t,x}^{2}$, clearly,
\begin{equation}
\left\Vert \widetilde{K}F\right\Vert _{L_{t}^{p}L_{x}^{q}}\leq\left\Vert Pe^{-itA}\right\Vert _{L^{2}\rightarrow L_{t}^{p}L_{x}^{q}}\left\Vert \int_{0}^{\infty}e^{isA}V_{1}\left(\cdot-\vec{Y}(s)\right)F(s)\,ds\right\Vert _{L^{2}}.\label{eq:TKF}
\end{equation}
The first factor on the right-hand side of \eqref{eq:TKF} is bounded
by Strichartz estimates for the free evolution. Consider the second
factor, by duality, it is sufficient to show
\begin{equation}
\left\Vert V_{1}\left(\cdot-\vec{Y}(t)\right)e^{-itA}\phi\right\Vert _{L_{t,x}^{2}}\lesssim\left\Vert \phi\right\Vert _{L^{2}},\,\forall\phi\in L^{2}\left(\mathbb{R}^{3}\right).
\end{equation}
By our assumption,
\[
\left|\vec{Y}(t)-\vec{\mu}t\right|\lesssim\left\langle t\right\rangle ^{-\beta},\,\beta>1,\,\left|\vec{\mu}\right|<1.
\]
Therefore, it reduces to show
\begin{equation}
\left\Vert \left(1+\left|x-\vec{\mu}t\right|\right)^{-\frac{1}{2}-\epsilon}e^{-itA}\phi\right\Vert _{L_{t,x}^{2}}\lesssim\left\Vert \phi\right\Vert _{L^{2}},\,\forall\phi\in L^{2}\left(\mathbb{R}^{3}\right).\label{eq:DAFT}
\end{equation}
Notice that this is a consequence of that the energy of the free wave
equation stays comparable under Lorentz transformations, Theorem \ref{thm:generalC}.
To show estimate \eqref{eq:DAFT}, one can apply the Lorentz transformation
$L$. In the new frame $\left(x',t'\right)$, then we can use the
standard local energy decay for free wave equations, estimate \eqref{eq:fullwave}
in Appendix B. Finally after applying an inverse transformation back
to the original frame, we obtain \eqref{eq:DAFT}.

From estimate \eqref{eq:DAFT}, one does have
\begin{equation}
\left\Vert V_{1}\left(\cdot-\vec{Y}(t)\right)e^{-itA}\phi\right\Vert _{L_{t,x}^{2}}\lesssim\left\Vert \phi\right\Vert _{L^{2}}.
\end{equation}
Therefore, indeed,
\begin{equation}
\left\Vert \widetilde{K}\right\Vert _{L_{t,x}^{2}\rightarrow L_{t}^{p}L_{x}^{q}}\leq C.
\end{equation}
Hence
\begin{equation}
\left\Vert P\int_{0}^{\infty}e^{-i\left(t-s\right)A}V_{1}V_{2}\left(\cdot-\vec{Y}(s)\right)u(s)\,ds\right\Vert _{L_{t}^{p}L_{x}^{q}}\lesssim\left\Vert V_{2}\left(x-\vec{Y}(s)\right)u\right\Vert _{L_{t,x}^{2}}.
\end{equation}
By our estimate \eqref{eq:moveweighted},
\begin{equation}
\left\Vert V_{2}\left(x-\vec{Y}(s)\right)u\right\Vert _{L_{t,x}^{2}}\lesssim\left(\int_{\mathbb{R}^{+}}\int_{\mathbb{R}^{3}}\frac{1}{\left\langle x-\vec{Y}(t)\right\rangle ^{\alpha}}\left|u(x,t)\right|^{2}dxdt\right)^{\frac{1}{2}}\lesssim\|f\|_{L^{2}}+\|g\|_{\dot{H}^{1}}.
\end{equation}
Therefore,
\begin{equation}
\left\Vert P\int_{0}^{\infty}e^{-i\left(t-s\right)A}V_{1}V_{2}\left(\cdot-\vec{Y}(s)\right)u(s)\,ds\right\Vert _{L_{t}^{p}L_{x}^{q}}\lesssim\|f\|_{L^{2}}+\|g\|_{\dot{H}^{1}}.
\end{equation}
Hence one can conclude
\begin{equation}
\left\Vert u\right\Vert _{L_{t}^{p}L_{x}^{q}}\lesssim\|f\|_{L^{2}}+\|g\|_{\dot{H}^{1}},
\end{equation}
as we claimed.
\end{proof}
The energy estimates in Theorem \ref{thm:Energy} can be established
in a similar manner.
\begin{proof}[Proof of Theorem \ref{thm:Energy}]
Again, we set $A=\sqrt{-\Delta}$ and notice that
\begin{equation}
\left\Vert Af\right\Vert _{L^{2}}\simeq\left\Vert f\right\Vert _{\dot{H}^{1}},\,\,\forall f\in C^{\infty}\left(\mathbb{R}^{3}\right).
\end{equation}
For real-valued $u=\left(u_{1},u_{2}\right)\in\mathcal{H}=\dot{H}^{1}\left(\mathbb{R}^{3}\right)\times L^{2}\left(\mathbb{\mathbb{R}}^{3}\right)$,
we write
\begin{equation}
U:=Au_{1}+iu_{2}.
\end{equation}
We also notice that $u$ solves the original equation if and only
if
\begin{equation}
U:=Au+i\partial_{t}u
\end{equation}
satisfies
\begin{equation}
i\partial_{t}U=AU+V\left(x-\vec{v}(t)\right)u,
\end{equation}
\begin{equation}
U(0)=Ag+if\in L^{2}\left(\mathbb{R}^{3}\right).
\end{equation}
By Duhamel's formula,
\begin{equation}
U(t)=e^{itA}U(0)-i\int_{0}^{t}e^{-i\left(t-s\right)A}\left(V\left(\cdot-\vec{Y}(s)\right)u(s)\right)\,ds.
\end{equation}
 From the energy estimate for the free evolution,
\begin{equation}
\sup_{t\in\mathbb{R}}\left\Vert e^{itA}U(0)\right\Vert _{L_{x}^{2}}\lesssim\left\Vert U(0)\right\Vert _{L^{2}}.\label{eq:Sfirst-1-3}
\end{equation}
Writing $V=V_{1}V_{2}$, it suffices to bound
\begin{equation}
\sup_{t\in\mathbb{R}}\left\Vert \int_{0}^{\infty}e^{-i\left(t-s\right)A}V_{1}V_{2}\left(\cdot-\vec{Y}(s)\right)u(s)\,ds\right\Vert _{L_{x}^{2}}.
\end{equation}
This is can be handled in a same manner as Theorem \ref{thm:Stri}.

It is clear that
\begin{equation}
\left\Vert \int_{0}^{\infty}e^{-i\left(t-s\right)A}V_{1}V_{2}\left(\cdot-\vec{Y}(s)\right)u(s)\,ds\right\Vert _{L_{t}^{\infty}L_{x}^{2}}\leq\left\Vert \widetilde{K}\right\Vert _{L_{t}^{2}L_{x}^{2}\rightarrow L_{t}^{\infty}L_{x}^{2}}\left\Vert V_{2}\left(x-\vec{Y}(t)\right)u\right\Vert _{L_{t}^{2}L_{x}^{2}},
\end{equation}
where
\begin{equation}
\left(\widetilde{K}F\right)(t):=\int_{0}^{\infty}e^{-i\left(t-s\right)A}V_{1}\left(\cdot-\vec{Y}(s)\right)F(s)\,ds.
\end{equation}
We need to estimate
\begin{equation}
\left\Vert \widetilde{K}\right\Vert _{L_{t}^{2}L_{x}^{2}\rightarrow L_{t}^{\infty}L_{x}^{2}}.
\end{equation}
Testing against $F\in L_{t}^{2}L_{x}^{2}$, clearly,
\begin{equation}
\left\Vert \widetilde{K}F\right\Vert _{L_{t}^{\infty}L_{x}^{2}}\leq\left\Vert e^{-itA}\right\Vert _{L^{2}\rightarrow L_{t}^{\infty}L_{x}^{2}}\left\Vert \int_{0}^{\infty}e^{isA}V_{1}\left(\cdot-\vec{Y}(s)\right)F(s)\,ds\right\Vert _{L^{2}}.\label{eq:TKF-1-3}
\end{equation}
The first factors on the right-hand side of \eqref{eq:TKF-1-3} is bounded
by the energy estimates for the free evolution. And the the remaining
steps are exactly the same as the proof above.

Therefore,we have
\begin{equation}
\sup_{t\geq0}\left(\|\nabla u(t)\|_{L^{2}}+\|u_{t}(t)\|_{L^{2}}\right)\lesssim\|f\|_{L^{2}}+\|g\|_{\dot{H}^{1}}
\end{equation}
as claimed.
\end{proof}
To finish this section, we show one important application of Theorem
\ref{thm:EndRStri} to establish the boundedness of the following
energy
\begin{equation}
E_{V}(t)=\int_{\mathbb{R}^{3}}\left|\nabla_{x}u\right|^{2}+\left|\partial_{t}u\right|^{2}+V\left(x-\vec{Y}(t)\right)\left|u\right|^{2}dx\label{eq:eneOM}
\end{equation}
as Corollary \ref{cor:ene}.
\begin{proof}[Corollary \ref{cor:ene}]
We might assume $u$ is smooth. Taking
the time derivative of $E_{V}(t)$ and by the fact that $u$ solves
equation, we obtain

\begin{equation}
\partial_{t}E_{V}(t)=\int_{\mathbb{R}^{3}}\partial_{t}V(x-\vec{Y}(t))\left|u(x,t)\right|^{2}dx=-\int_{\mathbb{R}^{3}}\vec{Y}'\left(t\right)\cdot\nabla V(y)\left|u^{S}(y,t)\right|^{2}dy.
\end{equation}
by a simple change of variable.

Note that
\begin{eqnarray}
\int_{0}^{\infty}\left|\partial_{t}E_{V}(t)\right|dt & \lesssim & \int_{0}^{\infty}\int_{\mathbb{R}^{3}}\left|\partial_{y}V(y)\right|\left|u^{S}(y)\right|^{2}dydt,\nonumber \\
 & \lesssim & \left\Vert \partial_{x}V\right\Vert _{L_{x}^{1}}\left\Vert u^{S}\right\Vert _{L_{x}^{\infty}L_{t}^{2}}^{2}\\
 & \lesssim & \left\Vert \left(g,f\right)\right\Vert _{\dot{H}^{1}\times L^{2}}^{2}\nonumber 
\end{eqnarray}
where in the last inequality, we applied Theorem \ref{thm:EndRStri}.

Therefore, for arbitrary $t\in\mathbb{R}^{+}$, we have
\begin{equation}
\left|E_{V}(t)-E_{V}(0)\right|\leq\int_{0}^{\infty}\left|\partial_{t}E_{V}(t)\right|dt\lesssim\left\Vert \left(g,f\right)\right\Vert _{\dot{H}^{1}\times L^{2}}^{2}
\end{equation}
which implies
\begin{equation}
\sup_{t}\left|E_{V}(t)\right|\lesssim\left\Vert \left(g,f\right)\right\Vert _{\dot{H}^{1}\times L^{2}}^{2}.
\end{equation}
We are done.
\end{proof}
With endpoint Strichartz estimates along smooth trajectories, we can
also derive inhomogenenous Strichartz estimates. One can find a detailed
argument in \cite{GC2}. 

\section{Scattering and Asymptotic Completeness\label{sec:Scattering}}

In this section, we show some applications of the results in this
paper. We will study the long-time behaviors for a scattering state
in the sense of Definition \ref{AO}. 

Following the notations from section above, we will still use the
short-hand notation
\begin{equation}
L_{t}^{p}L_{x}^{q}:=L_{t}^{p}\left([0,\infty),\,L_{x}^{q}\right).
\end{equation}
We reformulate the wave equation as a Hamiltonian system,
\begin{equation}
U'=JE'(U)
\end{equation}
where $J$ is a skew symmetric matrix and $E'(U)$ is the Frechet
derivative of the conserved quantity. Setting
\begin{equation}
U:=\left(\begin{array}{c}
u\\
\partial_{t}u
\end{array}\right),\,J:=\left(\begin{array}{cc}
0 & 1\\
-1 & 0
\end{array}\right),\,H_{F}:=\left(\begin{array}{cc}
-\Delta & 0\\
0 & 1
\end{array}\right),\label{eq:bigU}
\end{equation}
we can rewrite the free wave equation as
\begin{equation}
\dot{U}_{0}-JH_{F}U_{0}=0,
\end{equation}
with initial data
\begin{equation}
U_{0}[0]=\left(\begin{array}{c}
g_{0}\\
f_{0}
\end{array}\right).
\end{equation}
The solution of the free wave equation is given by
\begin{equation}
U_{0}=e^{tJH_{F}}U_{0}[0].
\end{equation}
\begin{thm}
\label{thm:scattering}Suppose $u$ is a scattering state in the sense
of Definition of \ref{AO} which solves
\begin{equation}
\partial_{tt}u-\Delta u+V(x-\vec{Y}(t))u=0\label{eq:scateq}
\end{equation}
with initial data
\begin{equation}
u(x,0)=g(x),\,u_{t}(x,0)=f(x).
\end{equation}
Write
\begin{equation}
U=\left(u,u_{t}\right)^{t}\in C^{0}\left([0,\infty);\,\dot{H}^{1}\right)\times C^{0}\left([0,\infty);\,L^{2}\right),
\end{equation}
with initial data \textup{$U[0]=\left(g,f\right)^{t}\in\dot{H}^{1}\times L^{2}$.
Then there exist free data
\[
U_{0}[0]=\left(g_{0},f_{0}\right)^{t}\in\dot{H}^{1}\times L^{2}
\]
 such that
\begin{equation}
\left\Vert U[t]-e^{tJH_{F}}U_{0}[0]\right\Vert _{\dot{H}^{1}\times L^{2}}\rightarrow0
\end{equation}
as $t\rightarrow\infty$.}
\end{thm}

\begin{proof}
We will still use the formulation in Theorem \ref{thm:Stri}. We set
$A=\sqrt{-\Delta}$ and notice that
\begin{equation}
\left\Vert Af\right\Vert _{L^{2}}\simeq\left\Vert f\right\Vert _{\dot{H}^{1}},\,\,\forall f\in C^{\infty}\left(\mathbb{R}^{3}\right).
\end{equation}
For real-valued $u=\left(u_{1},u_{2}\right)\in\mathcal{H}=\dot{H}^{1}\left(\mathbb{R}^{3}\right)\times L^{2}\left(\mathbb{\mathbb{R}}^{3}\right)$,
we write
\begin{equation}
U:=Au_{1}+iu_{2}.
\end{equation}
As before, then $U$ solves
\begin{equation}
i\partial_{t}U=AU+V\left(x-\vec{Y}(t)\right)u,
\end{equation}
\begin{equation}
U(0)=Ag+if\in L^{2}\left(\mathbb{R}^{3}\right).
\end{equation}
By Duhamel's formula, for fixed $T$
\begin{equation}
U(T)=e^{iTA}U(0)-i\int_{0}^{T}e^{-i\left(T-s\right)A}\left(V\left(\cdot-\vec{Y}(s)\right)u(s)\right)\,ds.
\end{equation}
Applying the free evolution backwards, we obtain
\begin{equation}
e^{-iTA}U(T)=U(0)-i\int_{0}^{T}e^{isA}\left(V\left(\cdot-\vec{Y}(s)\right)u(s)\right)\,ds.
\end{equation}
Letting $T$ go to $\infty$, we define
\begin{equation}
U_{0}(0):=U(0)-i\int_{0}^{\infty}e^{isA}\left(V\left(\cdot-\vec{Y}(s)\right)u(s)\right)\,ds
\end{equation}
By construction, we just need to show $U_{0}[0]$ is well-defined
in $L^{2}$, then automatically,
\begin{equation}
\left\Vert U(t)-e^{itA}U_{0}(0)\right\Vert _{L^{2}}\rightarrow0.
\end{equation}
 It suffices to show
\begin{equation}
\int_{0}^{\infty}e^{isA}\left(V\left(\cdot-\vec{Y}(s)\right)u(s)\right)\,ds\in L^{2}.
\end{equation}
Then following the argument as in the proof of Theorem \ref{thm:Stri},
we write $V=V_{1}V_{2}$.

We consider
\begin{equation}
\left\Vert \int_{0}^{\infty}e^{isA}V_{1}V_{2}\left(\cdot-\vec{Y}(s)\right)u(s)\,ds\right\Vert _{L_{x}^{2}}\leq\left\Vert K\right\Vert _{L_{t,x}^{2}\rightarrow L_{x}^{2}}\left\Vert V_{2}\left(\cdot-\vec{Y}(s)\right)u\right\Vert _{L_{t,x}^{2}},
\end{equation}
where
\begin{equation}
\left(KF\right)(t):=\int_{0}^{\infty}e^{isA}V_{1}\left(\cdot-\vec{Y}(s)\right)F(s)\,ds.
\end{equation}
By the same argument in the proof of Theorem \ref{thm:Energy}, one
has
\begin{equation}
\left\Vert K\right\Vert _{L_{t,x}^{2}\rightarrow L_{x}^{2}}\leq C.
\end{equation}
Therefore by estimate \eqref{eq:moveweighted},
\begin{equation}
\left\Vert V_{2}\left(x-\vec{Y}(t)\right)u\right\Vert _{L_{t,x}^{2}}\lesssim\left(\int_{\mathbb{R}^{+}}\int_{\mathbb{R}^{3}}\frac{1}{\left\langle x-\vec{Y}(t)\right\rangle ^{\alpha}}\left|u(x,t)\right|^{2}dxdt\right)^{\frac{1}{2}}\lesssim\|f\|_{L^{2}}+\|g\|_{\dot{H}^{1}}.
\end{equation}
Hence
\begin{equation}
\left\Vert \int_{0}^{\infty}e^{isA}\left(V\left(\cdot-\vec{Y}(s)\right)u(s)\right)\,ds\right\Vert _{L^{2}}\lesssim\|f\|_{L^{2}}+\|g\|_{\dot{H}^{1}}.
\end{equation}
So
\begin{equation}
U_{0}(0):=U(0)-i\int_{0}^{\infty}e^{isA}\left(V\left(\cdot-\vec{Y}(s)\right)u(s)\right)\,ds
\end{equation}
is well-defined in $L^{2}$ and
\begin{equation}
\left\Vert U(t)-e^{itA}U_{0}(0)\right\Vert _{L^{2}}\rightarrow0.
\end{equation}
Define
\begin{equation}
\left(g_{0},f_{0}\right):=\left(A^{-1}\Re U_{0}(0),\,\Im U_{0}(0)\right).
\end{equation}
By construction, notice that
\begin{equation}
U[t]=\left(A^{-1}\Re U(t),\,\Im U(t)\right)
\end{equation}
and
\begin{equation}
\left\Vert U[t]-e^{tJH_{F}}U_{0}[0]\right\Vert _{\dot{H}^{1}\times L^{2}}\rightarrow0.
\end{equation}
We are done.
\end{proof}
To finish this section, we show the asymptotic completeness for the
wave equation with the potential moving along a straight line:

\begin{equation}
\partial_{tt}u-\Delta u+V(x-\vec{\mu}t)u=0.
\end{equation}
For the asymptotic behavior of the trajectory, without loss of generality,
we still that assume $\vec{u}$ is along $\vec{e}_{1}$.

Let $m_{1},\,\ldots,\,m_{w}$ be the normalized bound states of
\begin{equation}
H=-\Delta+V\left(\sqrt{1-\left|\mu\right|^2}x_{1},x_{2},x_{3}\right)
\end{equation}
associated with eigenvalues $-\lambda_{1}^{2},\,\ldots,\,-\lambda_{w}^{2}$
respectively with $\lambda_{i}>0,\,i=1,\ldots,w$. Setting
\begin{equation}
A_{H}=\left(\begin{array}{cc}
0 & 1\\
-H & 0
\end{array}\right),
\end{equation}
then the point spectrum of $A_{H}$ is
\begin{equation}
\sigma_{p}=\bigcup_{i=1}^{w}\left\{ \pm\lambda_{i}\right\} 
\end{equation}
and the continuous spectrum is
\begin{equation}
\sigma_{c}=i\left(-\infty,\infty\right).
\end{equation}
Setting
\begin{equation}
E_{i}^{\pm}=\left(\begin{array}{c}
m_{i}\\
\pm\lambda_{i}m_{i}
\end{array}\right),\,i=1,\ldots,w,
\end{equation}
we know $E_{i}^{\pm}$ are eigenvectors of $A_{H}$ with eigenvalues
$\pm\lambda_{i}$. One can define the associated Riesz projection
\begin{equation}
P_{i,\pm}\left(H\right):=\left\langle \cdot,JE_{i}^{\mp}\right\rangle E_{i}^{\pm}\label{eq:rieszpro}
\end{equation}
onto $E_{i}^{\pm}$. One can check
\begin{equation}
P_{i,\pm}\left(H\right)\left(\begin{array}{c}
u\\
\partial_{t}u
\end{array}\right)=\left\langle \pm\lambda_{i}u(t)+\partial_{t}u(t),\,m_{i}\right\rangle .
\end{equation}
From the standard asymptotic completeness results, if we write
\begin{equation}
\dot{U}=A_{H}U,\ U=\left(\begin{array}{c}
u\\
\partial_{t}u
\end{array}\right)\ \text{and}\ U[0]=\left(\begin{array}{c}
g\\
f
\end{array}\right)
\end{equation}
then one can decompose the evolution as
\begin{equation}
U(t)=\sum_{i=1}^{w}\left\langle U[0],JE_{i,\mp}\right\rangle e^{\pm\lambda_{i}t}E_{i}^{\pm}+e^{tH_{F}}U_{0}[0]+R(t)\label{eq:StAC}
\end{equation}
where $e^{tH_{F}}U_{0}[0]$ is the free evolution with initial data
$U_{0}[0]$ and
\begin{equation}
\left\Vert R(t)\right\Vert _{\dot{H}^{1}\times L^{2}}\rightarrow0,\,\,t\rightarrow\infty.
\end{equation}
With notations above, we can obtain a similar decomposition as \eqref{eq:StAC}
when the potential is moving. 
\begin{cor}
\label{cor:ACOne}Suppose $H$ admits no eigenfunction nor resonances
at zero. Let $u$ solve
\begin{equation}
\partial_{tt}u-\Delta u+V(x-\vec{\mu}t)u=0.
\end{equation}
Write
\begin{equation}
U=\left(u,u_{t}\right)^{t}\in C^{0}\left([0,\infty);\,\dot{H}^{1}\right)\times C^{0}\left([0,\infty);\,L^{2}\right),
\end{equation}
 with initial data \textup{$U[0]=\left(g,f\right)^{t}\in\dot{H}^{1}\times L^{2}$.
Then there exist free data
\[
U_{0}[0]=\left(g_{0},f_{0}\right)^{t}\in\dot{H}^{1}\times L^{2}
\]
 such that with $\gamma=\frac{1}{\sqrt{1-\left|\mu\right|^{2}}}$
\begin{equation}
U(t)=\sum_{i=1}^{w}a_{i,\pm}e^{\pm\lambda_{i}\gamma\left(t-\mu x_{1}\right)}E_{i,\mu}^{\pm}\left(x,t\right)+e^{tH_{F}}U_{0}[0]+R(t)\label{eq:ACOne}
\end{equation}
where
\[
E_{i,\mu}^{\pm}\left(x,t\right)=E_{i}^{\pm}\left(\gamma\left(x_{1}-\mu t\right),x_{2},x_{3}\right)
\]
and
\begin{equation}
\left\Vert R(t)\right\Vert _{\dot{H}^{1}\times L^{2}}\rightarrow0,\,\,t\rightarrow\infty.
\end{equation}
}
\end{cor}

\begin{proof}
Applying a Lorentz transformation such that under the new frame $\left(x',t'\right)$,
$V$ is stationary, by the standard asymptotic completeness decomposition,
one can write
\begin{equation}
U_{L}\left(x',t'\right)=\sum_{i=1}^{w}\left\langle U_{L}[0],JE_{i,\mp}\right\rangle e^{\pm\lambda_{i}t'}E_{i}^{\pm}\left(x'\right)+\mathcal{R}_{L}\left(x',t'\right),\label{eq:Ldecomp}
\end{equation}
where again, we used subscript $L$ to denote the function under the
new frame. 

Clearly, by the decomposition above \eqref{eq:Ldecomp},
\begin{equation}
P_{b}\left(H\right)\mathcal{R}_{L}\left(x',t'\right)=0.
\end{equation}
Then in the original frame,
\begin{equation}
U(t)=\sum_{i=1}^{w}a_{i,\pm}e^{\pm\lambda_{i}\gamma\left(t-\mu x_{1}\right)}E_{i,\mu}^{\pm}\left(x,t\right)+\mathcal{R}\left(x,t\right).
\end{equation}
where
\begin{equation}
a_{i,\pm}=\left\langle U_{L}[0],JE_{i,\mp}\right\rangle .
\end{equation}
By construction, $\mathcal{R}(x,t)$ satisfies the conditions in Theorem
\ref{thm:scattering}. Hence
\begin{equation}
\mathcal{R}(x,t)=e^{tH_{F}}U_{0}[0]+R(t)
\end{equation}
where $e^{tH_{F}}U_{0}[0]$ is the free evolution with initial data
$U_{0}[0]$ and
\begin{equation}
\left\Vert R(t)\right\Vert _{\dot{H}^{1}\times L^{2}}\rightarrow0,\,\,t\rightarrow\infty.
\end{equation}
Therefore, finally, we can write
\begin{equation}
U(t)=\sum_{i=1}^{w}a_{i,\pm}e^{\pm\lambda_{i}\gamma\left(t-\mu x_{1}\right)}E_{i,\mu}^{\pm}\left(x,t\right)+e^{tH_{F}}U_{0}[0]+R(t)
\end{equation}
with
\[
E_{i,\mu}^{\pm}\left(x,t\right)=E_{i}^{\pm}\left(\gamma\left(x_{1}-\mu t\right),x_{2},x_{3}\right)
\]
and
\begin{equation}
\left\Vert R(t)\right\Vert _{\dot{H}^{1}\times L^{2}}\rightarrow0,\,\,t\rightarrow\infty.
\end{equation}
The claim is proved.
\end{proof}
\begin{rem}
\label{rem:asyC}As a final remark, we point out that there is no
hope to establish an elegant asymptotic completeness if the potential
is not moving along a straight line. If there is a perturbation from
that case, the interaction among bound states becomes complicated.
 Basically, the mechanism is that if the evolution of one bound state
is activated, say the bound state with the highest energy, then it
will not only cause exponential growth with highest rate for itself
but also make the evolution of other bound states grow exponentially.
Meanwhile, if we have a scattering state, the evolution of bound states
is controllable. But one can obtain an exponential dichotomy decomposition
for the general case and for the subexponential part, one can show
the scattering behavior, see \cite{CJ}.
\end{rem}
\appendix

\section{Pointwise decay}

For the sake of completeness, in this appendix, we provide the proof
of dispersive estimates for the free wave equation in $\mathbb{R}^{3}$
based on the idea of reversed Strichartz estimates.
\begin{thm}
\label{thm:dispersive}In $\mathbb{R}^{3}$, suppose $f\in L^{2},\,\nabla f\in L^{1}$
and $g\in L^{2},\,\Delta g\in L^{1}$. Then one has the following
estimates:
\begin{equation}
\left\Vert \frac{\sin\left(t\sqrt{-\Delta}\right)}{\sqrt{-\Delta}}f\right\Vert _{L_{x}^{\infty}}\lesssim\frac{1}{\left|t\right|}\left\Vert \nabla f\right\Vert _{L_{x}^{1}},
\end{equation}
\begin{equation}
\left\Vert \cos\left(t\sqrt{-\Delta}\right)g\right\Vert _{L_{x}^{\infty}}\lesssim\frac{1}{\left|t\right|}\left\Vert \Delta g\right\Vert _{L_{x}^{1}}.
\end{equation}
\end{thm}

\begin{rem*}
Note that the second estimate is slightly different from the estimates
commonly used in the literature. For example, in Krieger-Schlag \cite{KS}
one needs the $L^{1}$ norm of $D^{2}g$ instead of $\Delta g$.
\end{rem*}
\begin{proof}
First of all, we consider
\begin{equation}
\frac{\sin\left(t\sqrt{-\Delta}\right)}{\sqrt{-\Delta}}f.
\end{equation}
In $\mathbb{R}^{3}$, one has
\begin{equation}
\frac{\sin\left(t\sqrt{-\Delta}\right)}{\sqrt{-\Delta}}f=\frac{1}{4\pi t}\int_{\left|x-y\right|=t}f(y)\,dy.
\end{equation}
Without loss of generality, we assume $t\geq0$.

Multiplying $t$ and integrating, we obtain
\begin{eqnarray}
\int_{0}^{\infty}\left|t\frac{\sin\left(t\sqrt{-\Delta}\right)}{\sqrt{-\Delta}}f\right|dt & \lesssim & \int_{0}^{\infty}\int_{\mathbb{S}^{2}}\left|f(x+r\omega)\right|r^{2}\,d\omega dr\nonumber \\
 & \lesssim & \left\Vert f\right\Vert _{L_{x}^{1}}.
\end{eqnarray}
Therefore,
\begin{equation}
\left\Vert t\frac{\sin\left(t\sqrt{-\Delta}\right)}{\sqrt{-\Delta}}f\right\Vert _{L_{x}^{\infty}L_{t}^{1}}\lesssim\left\Vert f\right\Vert _{L_{x}^{1}}.
\end{equation}
Notice that, from the estimate above, we also have
\begin{equation}
\left\Vert \int_{t}^{\infty}\frac{\sin\left(s\sqrt{-\Delta}\right)}{\sqrt{-\Delta}}f\,ds\right\Vert _{L_{x}^{\infty}}\lesssim\frac{1}{\left|t\right|}\left\Vert f\right\Vert _{L_{x}^{1}}.
\end{equation}
Replacing $f$ with $\Delta f$, it implies that
\begin{equation}
\left\Vert \int_{t}^{\infty}\sqrt{-\Delta}\sin\left(s\sqrt{-\Delta}\right)f\,ds\right\Vert _{L_{x}^{\infty}}\lesssim\frac{1}{\left|t\right|}\left\Vert \Delta f\right\Vert _{L_{x}^{1}}.
\end{equation}
On the other hand,
\begin{eqnarray}
\int_{0}^{\infty}\left|t\cos\left(t\sqrt{-\Delta}\right)f\right|dt & \lesssim & \int_{0}^{\infty}\int_{\mathbb{S}^{2}}\left|rf\left(x+r\omega\right)d\omega+r^{2}\partial_{r}f\left(x+r\omega\right)\right|\,d\omega dr\nonumber \\
 & \lesssim & \left\Vert \nabla f\right\Vert _{L_{x}^{1}}
\end{eqnarray}
where in the last inequality, we applied integration by parts in $r$
in the first term of the RHS of the first line.

Therefore,
\begin{equation}
\left\Vert t\cos\left(t\sqrt{-\Delta}\right)f\right\Vert _{L_{x}^{\infty}L_{t}^{1}}\lesssim\left\Vert \nabla f\right\Vert _{L_{x}^{1}}.
\end{equation}
Hence
\begin{equation}
\left\Vert \int_{t}^{\infty}\cos\left(s\sqrt{-\Delta}\right)f\,ds\right\Vert _{L_{x}^{\infty}}\lesssim\frac{1}{\left|t\right|}\left\Vert \nabla f\right\Vert _{L_{x}^{1}}.
\end{equation}
Finally, we check
\begin{equation}
\frac{\sin\left(t\sqrt{-\Delta}\right)}{\sqrt{-\Delta}}f=\int_{t}^{\infty}\cos\left(s\sqrt{-\Delta}\right)f\,ds,
\end{equation}
and
\begin{equation}
\cos\left(t\sqrt{-\Delta}\right)g=\int_{t}^{\infty}\sqrt{-\Delta}\sin\left(s\sqrt{-\Delta}\right)g\,ds.
\end{equation}
It suffices to show expressions hold for tast functions. Let $f,\,g,\,h$ be any test functions. Define
\begin{equation}
Ag=\cos\left(t\sqrt{-\Delta}\right)g-\int_{t}^{\infty}\sqrt{-\Delta}\sin\left(s\sqrt{-\Delta}\right)g\,ds
\end{equation}
and
\begin{equation}
Bf=\frac{\sin\left(t\sqrt{-\Delta}\right)}{\sqrt{-\Delta}}f-\int_{t}^{\infty}\cos\left(s\sqrt{-\Delta}\right)f\,ds.
\end{equation}
It is easy to check that $A,\,B$ are independent of $t$ by taking the time derivative of the above expressions.

To see $A=B=0$, for $A$, one observes that
\begin{equation}
\left\langle \cos\left(t\sqrt{-\Delta}\right)g,\,h\right\rangle \rightarrow0
\end{equation}
and
\begin{equation}
\left\Vert \int_{t}^{\infty}\frac{\sin\left(s\sqrt{-\Delta}\right)}{\sqrt{-\Delta}}f\,ds\right\Vert _{L_{x}^{\infty}}\lesssim\frac{1}{\left|t\right|}\left\Vert f\right\Vert _{L_{x}^{1}}.
\end{equation}
Therefore,
\begin{equation}
\left\langle Ag,\,h\right\rangle \rightarrow0,\,t\rightarrow\infty.
\end{equation}
Since $A$ is independent of $t$, one concludes that
\begin{equation}
\left\langle Ag,\,h\right\rangle =0
\end{equation}
for any pair of test functions and hence
\begin{equation}
A=0.
\end{equation}
Similarly, we get
\begin{equation}
B=0.
\end{equation}
Therefore by our calculations above, we can obtain the dispersive
estimates for the free wave equation,
\begin{eqnarray}
\left\Vert \frac{\sin\left(t\sqrt{-\Delta}\right)}{\sqrt{-\Delta}}f\right\Vert _{L_{x}^{\infty}} & = & \left\Vert \int_{t}^{\infty}\cos\left(s\sqrt{-\Delta}\right)f\,ds\right\Vert _{L_{x}^{\infty}}\nonumber \\
 & \lesssim & \frac{1}{\left|t\right|}\left\Vert \nabla f\right\Vert _{L_{x}^{1}},
\end{eqnarray}
and
\begin{eqnarray}
\left\Vert \cos\left(t\sqrt{-\Delta}\right)g\right\Vert _{L_{x}^{\infty}} & = & \left\Vert \int_{t}^{\infty}\sqrt{-\Delta}\sin\left(s\sqrt{-\Delta}\right)g\,ds\right\Vert _{L_{x}^{\infty}}\nonumber \\
 & \lesssim & \frac{1}{\left|t\right|}\left\Vert \Delta g\right\Vert _{L_{x}^{1}}.
\end{eqnarray}
The theorem is proved.
\end{proof}

\section{Local energy decay}

We derive the local energy decay estimate for the free wave equation
by the Fourier method.

Recall the coarea formula: for a a real-valued Lipschitz function
$u$ and a $L^{1}$ function $g$ then
\begin{equation}
\int_{\mathbb{R}^{n}}g(x)\left|\nabla u(x)\right|dx=\int_{\mathbb{R}}\int_{\left\{ u(x)=t\right\} }g(x)\,d\sigma(x)dt,\label{eq:coarea}
\end{equation}
where $\sigma$ is the surface measure.
\begin{lem}
\label{lem:area}For $F\in C_{0}^{\infty}$, $\phi$ smooth and non-degenerate,i.e.
$\left|\nabla\phi(x)\right|\neq0$, one has
\begin{equation}
\int_{\mathbb{R}}\int_{\mathbb{R}^{n}}e^{i\lambda\phi(x)}F(x)\,dxd\lambda=\left(2\pi\right)^{n}\int_{\left\{ \phi=0\right\} }\frac{F(x)}{\left|\nabla\phi(x)\right|}\,d\sigma(x).\label{eq:area}
\end{equation}
\end{lem}

\begin{proof}
From \eqref{eq:coarea},
\begin{equation}
\int_{\mathbb{R}}\int_{\mathbb{R}^{n}}e^{i\lambda\phi(x)}F(x)\,dxd\lambda=\int_{\mathbb{R}}\int_{\mathbb{R}}e^{i\lambda y}\int_{\left\{ \phi=y\right\} }\frac{F(x)}{\left|\nabla\phi(x)\right|}\,d\sigma(x)dyd\lambda.
\end{equation}
Denote $\int_{\left\{ \phi=y\right\} }\frac{F(x)}{\left|\nabla\phi(x)\right|}\,d\sigma(x)=g(y)$,
then
\begin{eqnarray}
\int_{\mathbb{R}}\int_{\mathbb{R}^{n}}e^{i\lambda\phi(x)}F(x)\,dxd\lambda & = & \int_{\mathbb{R}}\int_{\mathbb{R}}e^{i\lambda y}g(y)\,dyd\lambda\nonumber \\
 & = & \left(2\pi\right)^{\frac{n}{2}}\int_{\mathbb{R}}\hat{g}(\lambda)\,d\lambda\nonumber \\
 & = & \left(2\pi\right)^{n}g(0)\nonumber \\
 & = & \int_{\left\{ \phi=0\right\} }\frac{F(x)}{\left|\nabla\phi(x)\right|}\,d\sigma(x).
\end{eqnarray}
We are done.
\end{proof}
It suffices to consider the half wave evolution,
\begin{equation}
e^{it\sqrt{-\Delta}}f.
\end{equation}
\begin{thm}[Local energy decay]
\label{thm:local}Let $\chi\geq0$ be a smooth
cut-off function such that $\hat{\chi}$ has compact support. Then
\begin{equation}
\left\Vert \chi(x)e^{it\sqrt{-\Delta}}f\right\Vert _{L_{t,x}^{2}}\lesssim\left\Vert f\right\Vert _{L_{x}^{2}}.\label{eq:local}
\end{equation}
\end{thm}

\begin{proof}
Consider
\begin{eqnarray}
\int_{\mathbb{R}}\int_{\mathbb{R}^{n}}\left|e^{it\sqrt{-\Delta}}f\right|^{2}(x)\chi(x)\,dxdt & = & \int_{\mathbb{R}}\left\langle e^{it\sqrt{-\Delta}}f,\chi(x)e^{it\sqrt{-\Delta}}f\right\rangle _{L^{2}}dt\nonumber \\
 & = & \int_{\mathbb{R}}\left\langle e^{it\left|\xi\right|}\hat{f}\left(\xi\right),\left[e^{it\left|\xi\right|}\hat{f}\left(\xi\right)\right]*\hat{\chi}\left(\xi\right)\right\rangle _{L^{2}}dt\\
 & = & \int_{\mathbb{R}}\int_{\mathbb{R}^{n}}\int_{\mathbb{R}^{n}}e^{it\left(\left|\xi\right|-\left|\eta\right|\right)}\hat{\chi}\left(\xi-\eta\right)\hat{f}\left(\xi\right)\hat{f}\left(\eta\right)\,d\eta d\xi dt.\nonumber 
\end{eqnarray}
Applying Lemma \ref{lem:area} with $\phi\left(\xi,\eta\right)=\left|\xi\right|-\left|\eta\right|$,
the surface $\left\{ \phi=0\right\} $ becomes $\left\{ \left|\xi\right|=\left|\eta\right|\right\} $
and $\left|\nabla\phi\right|=\sqrt{2}$. It follows that
\begin{eqnarray}
\int_{\mathbb{R}}\int_{\mathbb{R}^{n}}\left|e^{it\sqrt{-\Delta}}f\right|^{2}(x)\chi(x)\,dxdt & \simeq & \int_{\left|\xi\right|=\left|\eta\right|}\hat{\chi}\left(\xi-\eta\right)\hat{f}\left(\xi\right)\hat{f}\left(\eta\right)\,d\sigma\nonumber \\
 & \lesssim & \int_{\left|\xi\right|=\left|\eta\right|}\left|\hat{\chi}\left(\xi-\eta\right)\right|\left[\left|\hat{f}\left(\xi\right)\right|^{2}+\left|\hat{f}\left(\eta\right)\right|^{2}\right]\,d\sigma\nonumber \\
 & \lesssim & \int_{\mathbb{R}^{n}}\left|\hat{f}\left(\xi\right)\right|^{2}\int_{\left|\xi\right|=\left|\eta\right|}\left|\hat{\chi}\left(\xi-\eta\right)\right|\,d\sigma d\xi\nonumber \\
 & \lesssim & \sup_{\xi}\left|K\left(\xi\right)\right|\int_{\mathbb{R}^{n}}\left|\hat{f}\left(\xi\right)\right|^{2}d\xi\\
 & \lesssim & \int_{\mathbb{R}^{n}}\left|f(x)\right|^{2}dx.\nonumber 
\end{eqnarray}
It reduces to show that
\begin{equation}
K(\xi)=\int_{\left|\xi\right|=\left|\eta\right|}\left|\hat{\chi}\left(\xi-\eta\right)\right|\,d\sigma
\end{equation}
is bounded uniformly in $\xi$. Since $\hat{\chi}\left(\xi\right)$
decays fast, we have
\[
\left|\hat{\chi}(\xi)\right|\lesssim\left\langle \xi\right\rangle ^{-N}
\]
and
\[
\left|\hat{\chi}\left(\xi\right)\right|\lesssim\left|\xi\right|^{1-\epsilon-n},
\]
where as usual, $\left\langle \xi\right\rangle =\left(1+\left|\xi\right|^{2}\right)^{\frac{1}{2}}$. 

Note that
\begin{eqnarray}
K(\xi) & = & \int_{\left|\xi\right|=\left|\eta\right|}\left|\hat{\chi}\left(\xi-\eta\right)\right|\,d\sigma\\
 & = & \int_{\left|\zeta-\xi\right|=\left|\xi\right|}\left|\hat{\chi}\left(\zeta\right)\right|\,d\sigma\nonumber \\
 & \lesssim & \int_{\left|\zeta-\xi\right|=\left|\xi\right|,\left|\zeta\right|<1}\left|\hat{\chi}\left(\zeta\right)\right|\,d\sigma\nonumber \\
 &  & +\int_{\left|\zeta-\xi\right|=\left|\xi\right|,\left|\zeta\right|>1}\left|\hat{\chi}\left(\zeta\right)\right|\,d\sigma\nonumber \\
 & \lesssim & C(n)\nonumber 
\end{eqnarray}
which is uniformly bounded in $\xi$ and only depends on $n$.

Therefore, we can conclude
\begin{equation}
\left\Vert \chi(x)e^{it\sqrt{-\Delta}}f\right\Vert _{L_{t,x}^{2}}\lesssim\left\Vert f\right\Vert _{L_{x}^{2}}.
\end{equation}
We are done.
\end{proof}
With dyadic decomposition and weights, one has a global version of
the result above:
\begin{cor}
\label{cor:fullwave}$\forall\epsilon>0$, one has
\begin{equation}
\left\Vert \left(1+\left|x\right|\right)^{-\frac{1}{2}-\epsilon}e^{it\sqrt{-\Delta}}f\right\Vert _{L_{t,x}^{2}}\lesssim_{\epsilon}\left\Vert f\right\Vert _{L_{x}^{2}}.\label{eq:fullwave}
\end{equation}
\end{cor}

\begin{proof}
Let $\chi(x)$ from Theorem \ref{thm:local} be a smooth version of
$1_{B_{1}(0)}$, the indicator function of the unit ball. It follows
that
\begin{eqnarray}
\left\Vert \chi\left(2^{-j}x\right)e^{it\sqrt{-\Delta}}f\right\Vert _{L_{t,x}^{2}}\nonumber \\
=2^{\frac{jn}{2}}2^{\frac{j}{2}}\left\Vert \chi\left(x\right)\left(e^{it\sqrt{-\Delta}}f\right)\left(2^{j}t,2^{j}x\right)\right\Vert _{L_{t,x}^{2}}\nonumber \\
=2^{\frac{jn}{2}}2^{\frac{j}{2}}\left\Vert \chi\left(x\right)\left(e^{it\sqrt{-\Delta}}f\left(2^{j}\cdot\right)\right)\right\Vert _{L_{t,x}^{2}}\nonumber \\
\lesssim2^{\frac{jn}{2}}2^{\frac{j}{2}}\left\Vert f\left(2^{j}\cdot\right)\right\Vert _{L_{x}^{2}}\nonumber \\
\lesssim2^{\frac{j}{2}}\left\Vert f\right\Vert _{L_{x}^{2}}.
\end{eqnarray}
Notice that
\begin{equation}
\left(1+\left|x\right|\right)^{-\frac{1}{2}-\epsilon}\simeq\sum_{j\geq0}2^{-j\left(\frac{1}{2}+\epsilon\right)}\chi\left(2^{-j}x\right)
\end{equation}
then with our computations above, we can conclude that
\begin{equation}
\left\Vert \sum_{j\geq0}2^{-j\left(\frac{1}{2}+\epsilon\right)}\chi\left(2^{-j}x\right)e^{it\sqrt{-\Delta}}f\right\Vert _{L_{t,x}^{2}}\lesssim_{\epsilon}\left\Vert f\right\Vert _{L_{x}^{2}},
\end{equation}
and hence
\begin{equation}
\left\Vert \left(1+\left|x\right|\right)^{-\frac{1}{2}-\epsilon}e^{it\sqrt{-\Delta}}f\right\Vert _{L_{t,x}^{2}}\lesssim_{\epsilon}\left\Vert f\right\Vert _{L_{x}^{2}}.
\end{equation}
The corollary is proved.
\end{proof}

\section{Global existence}

In this appendix, we discuss the global existence of solutions to
the wave equation with time-dependent potentials. Lorentz transformations
are important tools in our analysis. Lorentz transformations are rotations
of space-time, therefore, a priori, one needs to show the global existence
of solutions to wave equations with time-dependent potentials.
\begin{thm}
\label{thm:globalexistence}Assume $V(x,t)\in L_{t,x}^{\infty}$.
Then for each $\left(g,\,f\right)\in H^{1}\left(\mathbb{R}^{3}\right)\times L^{2}\left(\mathbb{R}^{3}\right),$
there is a unique solution $\left(u,\,u_{t}\right)\in C\left(\mathbb{R},\,H^{1}\left(\mathbb{R}^{3}\right)\right)\times C\left(\mathbb{R},\,L^{2}\left(\mathbb{R}^{3}\right)\right)$
to
\begin{equation}
\partial_{tt}u-\Delta u+V(x,t)u=0\label{eq:timedepexist}
\end{equation}
with initial data
\begin{equation}
u(x,0)=g,\,\partial_{t}u(x,0)=f.
\end{equation}
\end{thm}

\begin{proof}
By Duhamel's formula, we might write the solution as
\begin{equation}
u=\frac{\sin\left(t\sqrt{-\Delta}\right)}{\sqrt{-\Delta}}f+\cos\left(t\sqrt{-\Delta}\right)g+\int_{0}^{t}\frac{\sin\left(\left(t-s\right)\sqrt{-\Delta}\right)}{\sqrt{-\Delta}}V(\cdot,s)u(s)\,ds.
\end{equation}
Starting from the local existence, we try to construct the solution
in
\begin{equation}
X=C\left([0,\,T],\,H^{1}\left(\mathbb{R}^{3}\right)\right)\times C\left([0,\,T),\,L^{2}\left(\mathbb{R}^{3}\right)\right)
\end{equation}
with $T\leq1$. One can view $u$ as the fixed-point of the map
\begin{equation}
S(h)(t)=\frac{\sin\left(t\sqrt{-\Delta}\right)}{\sqrt{-\Delta}}f+\cos\left(t\sqrt{-\Delta}\right)g+\int_{0}^{t}\frac{\sin\left(\left(t-s\right)\sqrt{-\Delta}\right)}{\sqrt{-\Delta}}V(\cdot,s)h(s)\,ds.
\end{equation}
Let
\begin{equation}
R=2\left\Vert \frac{\sin\left(t\sqrt{-\Delta}\right)}{\sqrt{-\Delta}}f+\cos\left(t\sqrt{-\Delta}\right)g\right\Vert _{X}.
\end{equation}
We will show when $T$ is small enough, $S$ will be a contraction
map in $B_{X}(0,R)$. 

Clearly,

\begin{equation}
\left\Vert S(h)(t)\right\Vert _{X}\leq\left\Vert \frac{\sin\left(t\sqrt{-\Delta}\right)}{\sqrt{-\Delta}}f+\cos\left(t\sqrt{-\Delta}\right)g\right\Vert _{X}+\left\Vert \int_{0}^{t}\frac{\sin\left(\left(t-s\right)\sqrt{-\Delta}\right)}{\sqrt{-\Delta}}V(\cdot,s)h(s)\,ds\right\Vert _{X}.
\end{equation}
By direct calculations,
\begin{equation}
\left\Vert \int_{0}^{t}\frac{\sin\left(\left(t-s\right)\sqrt{-\Delta}\right)}{\sqrt{-\Delta}}V(\cdot,s)h(s)\,ds\right\Vert _{L_{x}^{2}}\leq T^{2}\left\Vert V\left(\cdot,t\right)h(t)\right\Vert _{L^{2}},
\end{equation}
\begin{equation}
\left\Vert \int_{0}^{t}\frac{\sin\left(\left(t-s\right)\sqrt{-\Delta}\right)}{\sqrt{-\Delta}}V(\cdot,s)h(s)\,ds\right\Vert _{\dot{H}_{x}^{1}}\leq T\left\Vert V\left(\cdot,t\right)h(t)\right\Vert _{L^{2}},
\end{equation}
and
\begin{equation}
\left\Vert \partial_{t}\left(\int_{0}^{t}\frac{\sin\left(\left(t-s\right)\sqrt{-\Delta}\right)}{\sqrt{-\Delta}}V(\cdot,s)h(s)\,ds\right)\right\Vert _{L_{x}^{2}}\leq T\left\Vert V\left(\cdot,t\right)h(t)\right\Vert _{L^{2}}.
\end{equation}
Therefore, we can pick $T\left\Vert V\right\Vert _{L_{t,x}^{\infty}}<\frac{1}{10}$,
we have
\begin{equation}
\left\Vert S(h)(t)\right\Vert _{X}\leq\left\Vert \frac{\sin\left(t\sqrt{-\Delta}\right)}{\sqrt{-\Delta}}f+\cos\left(t\sqrt{-\Delta}\right)g\right\Vert _{X}+\frac{1}{2}\left\Vert h\right\Vert _{X}.
\end{equation}
Hence, $S$ maps $B_{X}\left(0,R\right)$ into itself. 

Next we show $S$ is a contraction. The calculations are straightforward.
\begin{equation}
\left\Vert S(h_{1}-h_{2})(t)\right\Vert _{X}\leq\left\Vert \int_{0}^{t}\frac{\sin\left(\left(t-s\right)\sqrt{-\Delta}\right)}{\sqrt{-\Delta}}V(\cdot,s)\left(h_{1}(s)-h_{2}(s)\right)\,ds\right\Vert _{X}.
\end{equation}
The the same arguments as above give
\begin{equation}
\left\Vert S(h_{1}-h_{2})(t)\right\Vert _{X}\leq\frac{1}{2}\left\Vert (h_{1}-h_{2})(t)\right\Vert _{X}.
\end{equation}
Therefore, by fixed point theorem, there is $u\in X$ such that
\begin{equation}
u=S(u),
\end{equation}
in other words, there exist $u\in C\left([0,\,T],\,H^{1}\left(\mathbb{R}^{3}\right)\right)\times C\left([0,\,T),\,L^{2}\left(\mathbb{R}^{3}\right)\right)$
such that
\begin{equation}
u=\frac{\sin\left(t\sqrt{-\Delta}\right)}{\sqrt{-\Delta}}f+\cos\left(t\sqrt{-\Delta}\right)g+\int_{0}^{t}\frac{\sin\left(\left(t-s\right)\sqrt{-\Delta}\right)}{\sqrt{-\Delta}}V(\cdot,s)u(s)\,ds.
\end{equation}

We notice that the choice of $T$ is independent of the size of the
initial data. Then we can repeat the argument above with $\left(u(T),\partial_{t}u(T)\right)$
as initial condition to construct the solution from $T$ to $2T$.
Iterating this process, one can easily construct the solution $\left(u,\,u_{t}\right)\in C\left(\mathbb{R},\,H^{1}\left(\mathbb{R}^{3}\right)\right)\times C\left(\mathbb{R},\,L^{2}\left(\mathbb{R}^{3}\right)\right)$.

Finally, we notice the uniqueness of the solution follows from Gr\"onwall's
inequality. Suppose one has two solutions $u_{1}$ and $u_{2}$ to
our equation with the same data, then
\begin{equation}
\left\Vert u_{1}-u_{2}\right\Vert _{H^{1}\times L^{2}}(t)\leq\int_{0}^{t}\left(t-s\right)\left\Vert u_{1}-u_{2}\right\Vert (s)\,ds.
\end{equation}
Applying Gr\"onwall's inequality over $[0,T]$, we obtain
\begin{equation}
\left\Vert u_{1}-u_{2}\right\Vert _{X}=0,
\end{equation}
which means $u_{1}\equiv u_{2}$ on $[0,T]$. Then by the same iteration
argument as above, we can conclude that in $C\left(\mathbb{R},\,H^{1}\left(\mathbb{R}^{3}\right)\right)\times C\left(\mathbb{R},\,L^{2}\left(\mathbb{R}^{3}\right)\right)$
\begin{equation}
u_{1}\equiv u_{2}.
\end{equation}
 Therefore, one obtains the uniqueness.

The theorem is proved. 
\end{proof}
In our setting, $V(x,t)=V\left(x-\vec{v}(t)\right)$ satisfies the
assumption of Theorem \ref{thm:globalexistence}, therefore we have
the global existence and uniqueness.
\begin{cor}
\label{cor:GlobalCharge}For each $\left(g,\,f\right)\in H^{1}\left(\mathbb{R}^{3}\right)\times L^{2}\left(\mathbb{R}^{3}\right),$
there is a unique global solution $\left(u,\,u_{t}\right)\in C\left(\mathbb{R},\,H^{1}\left(\mathbb{R}^{3}\right)\right)\times C\left(\mathbb{R},\,L^{2}\left(\mathbb{R}^{3}\right)\right)$
to the wave equation
\begin{equation}
\partial_{tt}u-\Delta u+V\left(x-\vec{v}(t)\right)u=0\label{eq:chargeexist}
\end{equation}
with initial data
\begin{equation}
u(x,0)=g,\,\partial_{t}u(x,0)=f.
\end{equation}
\end{cor}

\begin{rem*}
The theorem above also applies to the charge transfer model in \cite{GC2}:
\[
\partial_{tt}u-\Delta u+\sum_{i=1}^{m}\sum_{j=1}^{m}V_{v_{j}}\left(x-\vec{v}_{j}t\right)u=0.
\]
\end{rem*}

\section{Revered Strichartz estimates}

In this appendix, we present an alternative approach to the homogeneous
endpoint reversed Strichartz estimates based on the Fourier transformation. 

We only consider $\frac{\sin\left(t\sqrt{-\Delta}\right)}{\sqrt{-\Delta}}f=\frac{1}{2}\frac{e^{it\sqrt{-\Delta}}}{\sqrt{-\Delta}}f-\frac{1}{2}\frac{e^{-it\sqrt{-\Delta}}}{\sqrt{-\Delta}}f$.
We can further reduce to consider
\begin{equation}
\frac{e^{it\sqrt{-\Delta}}}{\sqrt{-\Delta}}f
\end{equation}
 With Fourier transform and polar coordinates $\xi=\lambda\omega$,
we have
\begin{eqnarray}
\frac{e^{it\sqrt{-\Delta}}}{\sqrt{-\Delta}}f & = & \int_{0}^{\infty}\int_{\mathbb{S}^{2}}\frac{e^{2\pi it\lambda}}{\lambda}e^{2\pi i\lambda\left(\omega\cdot x\right)}\lambda^{2}\hat{f}\left(\lambda\omega\right)d\omega d\lambda\nonumber \\
 & = & \int_{\mathbb{R}}e^{2\pi it\lambda}\left(\chi_{[0,\infty)}(\lambda)\int_{\mathbb{S}^{2}}e^{2\pi i\lambda\left(\omega\cdot x\right)}\lambda\hat{f}\left(\lambda\omega\right)d\omega\right)d\lambda\nonumber \\
 & = & \int_{\mathbb{R}}e^{2\pi it\lambda}G(x,\lambda)d\lambda
\end{eqnarray}
where
\begin{equation}
G(x,\lambda)=\chi_{[0,\infty)}(\lambda)\int_{\mathbb{S}^{2}}e^{2\pi i\lambda\left(\omega\cdot x\right)}\lambda\hat{f}\left(\lambda\omega\right)d\omega.
\end{equation}
By Plancherel's Theorem, we know for fixed $x$,
\begin{equation}
\left\Vert \frac{e^{it\sqrt{-\Delta}}}{\sqrt{-\Delta}}f\right\Vert _{L_{t}^{2}}=\left\Vert G(x,\lambda)\right\Vert _{L_{\lambda}^{2}}.
\end{equation}
\begin{eqnarray}
G^{2}(x,\lambda) & = & \left(\chi_{[0,\infty)}(\lambda)\int_{\mathbb{S}^{2}}e^{2\pi i\lambda\left(\omega\cdot x\right)}\lambda\hat{f}\left(\lambda\omega\right)d\omega\right)^{2}\nonumber \\
 & \lesssim & \chi_{[0,\infty)}(\lambda)\int_{\mathbb{S}^{2}}\lambda^{2}\left|\hat{f}\left(\lambda\omega\right)\right|^{2}d\omega
\end{eqnarray}
\begin{eqnarray}
\left\Vert \frac{e^{it\sqrt{-\Delta}}}{\sqrt{-\Delta}}f\right\Vert _{L_{t}^{2}}^{2} & \lesssim & \int_{0}^{\infty}\int_{\mathbb{S}^{2}}\lambda^{2}\left|\hat{f}\left(\lambda\omega\right)\right|^{2}d\omega d\lambda\nonumber \\
 & \lesssim & \int\left|\hat{f}\left(\xi\right)\right|^{2}d\xi\\
 & = & \int\left|f\left(x\right)\right|^{2}dx.\nonumber 
\end{eqnarray}
Therefore,
\begin{equation}
\left\Vert \frac{\sin\left(t\sqrt{-\Delta}\right)}{\sqrt{-\Delta}}f\right\Vert _{L_{x}^{\infty}L_{t}^{2}}\lesssim\left\Vert f\right\Vert _{L^{2}}
\end{equation}
as desired.
\begin{rem*}
The two dimension version was obtained in \cite{Oh} and is mentioned
in \cite{B}:
\begin{equation}
\left\Vert e^{it\sqrt{-\Delta}}f\right\Vert _{L_{x}^{\infty}L_{t}^{2}}\lesssim\left\Vert f\right\Vert _{\dot{B}_{2,1}^{1/2}}.
\end{equation}
\end{rem*}


\begin{thebibliography}{BecGo}
\bibitem[Agm]{Agm} Agmon S. Spectral properties of Schrödinger operators
and scattering theory.\emph{ Ann. Scuola Norm. Sup. Pisa Cl. Sci.}
(4) 2 (1975), no.~2, 151\textendash 218.

\bibitem[BecGo]{BecGo} Beceanu, M. and Goldberg, M. Strichartz estimates
and maximal operators for the wave equation in $\mathbb{R}^{3}$.
\emph{J. Funct. Anal. }266 (2014), no.~3, 1476\textendash 1510. 

\bibitem[Bou]{Bou}Bourgain, J. \emph{Global solutions of nonlinear
Schrödinger equations}. \emph{American Mathematical Society Colloquium
Publications, }46. American Mathematical Society, Providence, RI,
1999. viii+182 pp.

\bibitem[Bec1]{Bec1} Beceanu, M. Structure of wave operators for
a scaling-critical class of potentials. \emph{Amer. J. Math.} 136
(2014), no. 2, 255\textendash 308. 

\bibitem[Bec2]{Bec2} Beceanu, M. Personal communication.

\bibitem[Bec3]{Bec3} Beceanu, M. New estimates for a time-dependent
Schrödinger equation. \emph{Duke Math. J}. 159 (2011), no. 3, 417\textendash 477.

\bibitem[B]{B} Beceanu, M. Decay estimates for the wave equation
in two dimensions. \emph{J. Differential Equations }260 (2016), no.
6, 5378\textendash 5420. 

\bibitem[BeSch]{BeSch} Beceanu, M. and Schlag, W. Structure formulas
for wave operators. Preprint (2016), arXiv:1612.07304.

\bibitem[BS]{BS} Beceanu, M. and Soffer, A. The Schrödinger equation
with a potential in rough motion. \emph{Comm. Partial Differential
Equations }37 (2012), no. 6, 969\textendash 1000.

\bibitem[CRT]{CRT} Cassani, D.; Ruf, B. and Tarsi, C. Optimal Sobolev type inequalities in Lorentz spaces. \emph{Potential Anal.} 39 (2013), no. 3, 265--285.

\bibitem[GC1]{GC1} Chen, G. Strichartz estimates for charge transfer
models. \emph{Discrete Contin. Dyn. Syst}. 37 (2017), no. 3, 1201-1226.

\bibitem[GC2]{GC2} Chen, G. Strichartz estimates for wave equtions
with charge transfer Hamiltonian. Preprint (2016), arXiv:1610.05226.

\bibitem[GC3]{GC3} Chen, G. Multisolitons for the defocusing energy
critical wave equation with potentials. \emph{Comm. Math. Phys. }364
(2018), no. 1, 45\textendash 82.

\bibitem[CJ]{CJ} Chen, G. and Jendrej, J. Lyapunov-type characterisation
of exponential dichotomies with applications to the heat and Klein-Gordon
equations. arXiv: 1812.07322 .

\bibitem[GV]{GV} Ginibre, J. and Velo, G. Generalized Strichartz
inequalities for the wave equation. \emph{J. Funct. Anal.} 133 (1995),
no.~1, 50\textendash 68.

\bibitem[Graf]{Graf}Graf, J. M. Phase Space analysis of the charge
transfer. \emph{Model. Helv. Physica Acta }63 (1990), 107\textendash 138.

\bibitem[JLSX]{JLSX}Jia, H., Liu, B.P., Schlag, W. and Xu, G.X. Generic
and non-generic behavior of solutions to the defocusing energy critical
wave equation with potential in the radial case. Preprint (2015),
arXiv:1506.04763. 

\bibitem[KT]{KT} Keel, M. and Tao, T. Endpoint Strichartz estimates.
\emph{Amer. J. Math. }120 (1998), no.~5, 955\textendash 980. 

\bibitem[KM]{KM} Klainerman, S. and Machedon, M. Space-time estimates
for null forms and the local existence theorem. \emph{Comm. Pure Appl.
Math}. 46 (1993), no.~9, 1221\textendash 1268. 

\bibitem[KS]{KS} Krieger J. and Schlag W. On the focusing critical
semi-linear wave equation. \emph{Amer. J. Math.} 129 (2007), no.~3,
843\textendash 913. 

\bibitem[MNNO]{MNNO} Machihara, S., Nakamura, M., Nakanishi, K. and
Ozawa, T. Endpoint Strichartz estimates and global solutions for the
nonlinear Dirac equation. \emph{J. Funct. Anal}. 219 (2005), no.~1,
1\textendash 20. 

\bibitem[LSch]{LSch} Lawrie, A. and Schlag, W. Scattering for wave
maps exterior to a ball. \emph{Adv. Math.} 232 (2013), 57\textendash 97.

\bibitem[MS]{MS} Muscalu, C. and Schlag, W. \emph{Classical and multilinear
harmonic analysis. Vol. I. Cambridge Studies in Advanced Mathematics},
138. Cambridge University Press\emph{, Cambridge,} 2013. xvi+324 pp.

\bibitem[NS]{NS}Nakanishi, K. and Schlag, W. \emph{Invariant manifolds
and dispersive Hamiltonian evolution equations. Zurich Lectures in
Advanced Mathematics}. European Mathematical Society (EMS), Zürich,
2011. vi+253 pp.

\bibitem[NS2]{NS2}Nakanishi, K.; Schlag, W. Global dynamics above
the ground state for the nonlinear Klein-Gordon equation without a
radial assumption. \emph{Arch. Ration. Mech. Anal. }203 (2012), no.
3, 809\textendash 851

\bibitem[Oh]{Oh} Oh, S-J. A reversed Strichartz Estimate in $\mathbb{R}^{1+2}$.

\bibitem[RS]{RS} Rodnianski, I. and Schlag, W. Time decay for solutions
of Schrödinger equations with rough and time-dependent potentials.
\emph{Invent. Math. }155 (2004), no.~3, 451\textendash 513.

\bibitem[RSS]{RSS} Rodnianski, I., Schlag, W. and Soffer, A. Dispersive
analysis of charge transfer models. \emph{Comm. Pure Appl. Math. }58
(2005), no.~2, 149\textendash 216.

\bibitem[RSS2]{RSS2}Rodnianski, I., Schlag, W. and Soffer, A. Asymptotic
stability of N-soliton states of NLS. preprint (2003), arXiv preprint
math/0309114.

\bibitem[Sch]{Sch} Schlag, W. Dispersive estimates for Schrödinger
operators: a survey. \emph{Mathematical aspects of nonlinear dispersive
equations,} 255\textendash 285, Ann. of Math. Stud., 163, Princeton
Univ. Press, Princeton, NJ, 2007.

\bibitem[Tao]{Tao} Tao, T. \emph{Nonlinear dispersive equations.
Local and global analysis. CBMS Regional Conference Series in Mathematics},
106. Published for the Conference Board of the Mathematical Sciences,
Washington, DC; by the American Mathematical Society, Providence,
RI, 2006. xvi+373 pp. 

\bibitem[Tar]{Tar} Tartar, L. Imbedding theorems of Sobolev spaces into Lorentz spaces. \emph{Boll. Unione Mat. Ital.
Sez. B Artic. Ric. Mat}.(8) 1 (1998), no. 3, 479--500.
\bibitem[Ya]{Ya} Yajima, K. The $W^{k,p}$continuity of wave operators
for Schrödinger operators. \emph{J. Math. Soc. Japan }47 (1995), no.\~3,
551\textendash 581.
\end{thebibliography}
\end{document}